\documentclass[a4paper,11pt]{elsarticle}

\usepackage[utf8]{inputenc}
\usepackage[T1]{fontenc}
\usepackage[english]{babel}

\usepackage[margin=1.5in]{geometry}
\usepackage{amsmath}
\usepackage{amssymb}
\usepackage{amsthm}
\usepackage{thmtools} 
\usepackage{mathtools}
\mathtoolsset{showonlyrefs, showmanualtags}
\usepackage{algorithm}
\usepackage[noend]{algpseudocode}
\usepackage{tikz}
\usepackage{tikz-cd}
\usepackage{multirow}
\usepackage{refcount}
\usepackage{bussproofs}
\usepackage{bm}
\usepackage[nottoc,notlot,notlof]{tocbibind}

\usepackage[final]{hyperref}

\usetikzlibrary{arrows, chains, positioning, decorations.pathreplacing, matrix, calc, fit}

\newtheorem{theorem}{Theorem}
\newtheorem{proposition}[theorem]{Proposition}

\newtheorem{lemma}[theorem]{Lemma}

\theoremstyle{definition}
\newtheorem{definition}[theorem]{Definition}
\newtheorem{example}[theorem]{Example}

\theoremstyle{remark}
\newtheorem{remark}[theorem]{Remark}

\algdef{SE}[DOWHILE]{Do}{doWhile}{\algorithmicdo}[1]{\algorithmicwhile\ #1}%
\renewcommand{\algorithmicrequire}{{\bf Input:}}
\renewcommand{\algorithmicensure}{{\bf Output:}}

\def\AA{\mathcal A}

\def\II{\mathcal I}
\def\Id{I}

\def\N{\mathbb N}
\def\Z{\mathbb Z}

\def\Formulas{\mathbf{Form}}
\def\Sent{\mathbf{Sent}}
\def\Term{\mathbf{Term}}
\def\Ground{\mathbf{Ground}}
\def\ObS{\mathbf{Ob}}

\def\Var{\mathbf{Var}}
\def\Sort{\mathbf{Sort}}
\def\Con{\mathbf{Con}}
\def\Fun{\mathbf{Fun}}
\def\Zero{\mathbf{Zero}}
\def\Arith{\mathbf{Arith}}
\def\Free{\mathbf{Free}}
\DeclareMathOperator{\CNF}{\textnormal{CNF}}
\DeclareMathOperator{\acker}{\mathcal{T}}
\DeclareMathOperator{\Axioms}{\mathcal{A}}

\DeclareMathOperator{\assign}{\mathfrak{a}}
\def\<#1>{\langle#1\rangle}
\DeclareMathOperator{\Ob}{Ob}
\DeclareMathOperator{\Mor}{Mor}

\begin{document}

\title{Universal truth of operator statements\\ via ideal membership}

 \author[1]{Clemens Hofstadler\corref{cor1}}
\ead{clemens.hofstadler@jku.at}
\author[2]{Clemens G.~Raab}
\ead{clemens.raab@ricam.oeaw.ac.at}
\author[3]{Georg Regensburger}
\ead{regensburger@mathematik.uni-kassel.de}

\cortext[cor1]{Corresponding author}
\affiliation[1]{organization={Johannes Kepler University},
		addressline={Altenberger Stra{\ss}e 69},
		postcode={4040 Linz}, 
		country={Austria}}
\affiliation[2]{organization={Johann Radon Institute for Computational and Applied Mathematics (RICAM), Austrian Academy of Sciences},
	  addressline={Altenberger Stra\ss e 69},
           postcode={4040 Linz},
           country={Austria}} 
\affiliation[3]{organization={University of Kassel},
		addressline={Heinrich-Plett-Stra\ss e 40},
		postcode={34132 Kassel}, 
		country={Germany}}

\begin{abstract}
We introduce a framework for proving statements about linear operators by verification of ideal membership in a free algebra.
More specifically, arbitrary first-order statements about identities of morphisms in preadditive semicategories can be treated.
We present a semi-decision procedure for validity of such formulas based on computations with noncommutative polynomials.
These algebraic computations automatically incorporate linearity and benefit from efficient ideal membership procedures.

In the framework, domains and codomains of operators are modelled using many-sorted first-order logic.
To eliminate quantifiers and function symbols from logical formulas, we apply Herbrand's theorem and Ackermann's reduction.
The validity of the resulting formulas is shown to be equivalent to finitely many ideal memberships of noncommutative polynomials.
We explain all relevant concepts and discuss computational aspects.
Furthermore, we illustrate our framework by proving concrete operator statements assisted by our computer algebra software. 
\end{abstract}



\maketitle

\tableofcontents

\section{Introduction}

Linear operators appear in several forms in many different settings all across mathematics.
They can be ring elements (as in $C^*$-algebras), but also (rectangular) matrices, or, more generally,  vector space and module homomorphisms (like bounded linear operators on Hilbert spaces).
Many statements in homological algebra, typically expressed in the language of abelian categories, also describe properties of linear operators.
In this work, we develop a framework to efficiently prove the validity of first-order statements about operator identities phrased in any of these settings by verifying ideal membership of noncommutative polynomials in a free algebra.


Translating operator identities into noncommutative polynomials provides several computational advantages.
Naturally, polynomial computations respect linearity. 
Furthermore, as our framework shows, restrictions imposed by domains and codomains of operators can be ignored when performing polynomial arithmetic.
Finally, one can exploit efficient computer algebra implementations and heuristics for polynomial computations, in particular, for verifying ideal membership.
These routines can also compute certificates for ideal membership that serve as an independent and easily verifiable proof.


Noncommutative polynomials have been used previously to model and analyse operator identities.
Starting from the pioneering work~\cite{helton1994rules,helton1998computer}, where polynomial techniques are used to simplify matrix identities in linear systems theory,
over~\cite{HS99}, where similar methods are used to discover operator identities and to solve matrix equations,
until~\cite{SL20,schmitz-master}, where proving operator identities via polynomial computations and related questions are addressed. 
Two of the authors also developed a framework that allows to infer the validity of a statement about operators
from ideal membership of noncommutative polynomials~\cite{RRH21}.
This framework can treat propositional statements where several identities (the assumptions) imply another identity (the claim).   


Our approach extends this framework, which was the motivation for the present work and builds the foundation for our results,
 in several ways.
First of all, we allow more general first-order statements that can include quantifiers, function symbols, and all boolean connectives.
We note that none of the previous algebraic approaches mentioned above can handle such general statements.
Secondly, while~\cite{RRH21} only gives a sufficient algebraic condition for an operator statement to hold, we provide with our main result (Theorem~\ref{thm nullstellensatz}) an equivalence between universal truth of operator statements and ideal membership of noncommutative polynomials.
This allows us to state with Procedure~\ref{algo semantic consequence} a semi-decision procedure for verifying the validity of operator statements based on 
polynomial computations.


In our framework, operator statements are treated as statements about \emph{morphisms} in \emph{preadditive semicategories}.
Semicategories are categories without identity morphisms.
They provide a natural and very general environment for our application, prescribing only linearity as a structural constraint.
We give an introduction to semicategories in Section~\ref{sec:algebraic-preliminaries} and show how all the aforementioned settings (rings, matrices, homomorphisms, abelian categories) can be considered as special instances of preadditive semicategories.
In that section, we also recall all the fundamentals of noncommutative polynomials in free algebras required in this work.


Before formally introducing our algebraic framework, we provide in Section~\ref{sec:uniqueness} an informal illustration of its usage. 
The goal of this section is to give an intuitive understanding of the framework,
demonstrating how basic operator statements can be proven by polynomial computations and that results about abstract morphisms allow to infer statements about linear operators in concrete settings.
In this way, readers can familiarise themselves with some of the core concepts of this work, such as noncommutative polynomials and morphisms in semicategories.


To formally model statements about morphisms in semicategories, we use \emph{many-sorted first-order logic}.
It extends classical first-order logic by assigning a \emph{sort} to each symbol.
These sorts allow to model objects from different universes and restrict which expressions can be formed.
For completeness, we provide an introduction to many-sorted first-order logic in~\ref{sec:many-sorted-logic}, see also~\cite{Man93} for a textbook exposition.

In Section~\ref{sec:modelling-operator-statements}, we describe how many-sorted logic can be used to model statements about morphisms and linear operators.
In particular, our logic consists of many-sorted first-order formulas with equality as the only predicate symbol, where sorts are used to represent domains and codomains of operators.
We call any well-formed formula in this logic an \emph{operator statement}.
In our language, we can, in particular, express the axioms of preadditive semicategories.
This allows us to specify the notion of \emph{universal truth} of operator statements via the logical concept of semantic consequence of these axioms.


G\"odel's famous completeness theorem~\cite{godel1930vollstandigkeit} for first-order logic states that semantic consequences can be validated by \emph{formal computations}.
A formal computation is a purely syntactic procedure performed on the formulas using a deductive system. 
Such deductive systems have to treat linearity axioms separately and ensure that all formed expressions respect the restrictions imposed by the sorts.
We provide an example of a deductive system in~\ref{sec:formal-comp}, namely a sequent calculus for many-sorted first-order logic with equality.
This sequent calculus will be crucial for proving our main result (Theorem~\ref{thm nullstellensatz}).


In Section~\ref{sec:arithmetic-statements}, we explain how a formal computation can be replaced by a computation with noncommutative polynomials.
In contrast to deductive systems, these polynomial computations naturally include linearity and, as will be shown, restrictions induced by sorts can be neglected.
When translating operator statements into polynomial statements, quantifiers and function symbols have to be treated separately. 
To this end, we recall two important concepts from the field of automated theorem proving, namely \emph{Herbrand's theorem}~\cite{Her30,buss1994herbrand} and \emph{Ackermann's reduction}~\cite{Ack54} in Section~\ref{sec:herbrand-ackermann}.
These procedures allow to reduce arbitrary operator statements to \emph{arithmetic operator statements}, containing no quantifiers and only arithmetic function symbols.
This restricted class of formulas allows for a direct translation into algebraic statements about noncommutative polynomials and builds the foundation for treating arbitrary operator statements.


To characterise universal truth of arithmetic operator statements, we introduce the process of \emph{idealisation} (Section~\ref{sec:idealisation}), which assigns to every such formula a predicate about polynomial ideal membership.
We then show in Section~\ref{sec:main-theorem} that this idealisation is true if and only if the corresponding operator statement is universally true.
In this way, universal truth of operator statements is reduced to ideal membership of noncommutative polynomials.




Since the fragment of first-order logic that we consider is only semi-decidable, we cannot hope to obtain a terminating algorithm for deciding universal truth of first-order operator statements.
Nevertheless, our results allow to present a semi-decision procedure (Procedure~\ref{algo semantic consequence}) that terminates if and only if an operator statement is universally true.
Our procedure is an adaptation of Gilmore's algorithm~\cite{gilmore1960proof} interleaved with ideal membership verifications.
We introduce this semi-decision procedure in Section~\ref{sec:all-operator-statements} and discuss computational aspects (Section~\ref{sec computational aspects}) that are relevant for obtaining an efficient procedure that can treat large, non-trivial examples.
Furthermore, we present in Section~\ref{sec:AE-quasi-identities} how Procedure~\ref{algo semantic consequence} simplifies for the class of \emph{arithmetic $\forall\exists$-quasi identities},
a class of operator statements that regularly appears in practical applications.

Finally, in Section~\ref{sec:applications}, we present several applications of our framework and show its capabilities.
First, we illustrate our methods on an introductory example about the existence of the Moore-Penrose inverse.
We also discuss a case study on statements regarding the Moore-Penrose inverse conducted by two of the authors~\cite{casc-2023}.
Using the framework developed in this work in combination with our computer algebra package \texttt{operator\_gb}
\footnote{available for \textsc{SageMath} at \url{https://github.com/ClemensHofstadler/operator_gb}
and for \textsc{Mathematica} at \url{https://github.com/ClemensHofstadler/OperatorGB}}~\cite{HRR19,Hof23} to perform the polynomial computations,
the proofs of a variety of operator statements could be automated, ranging from classical facts to recent results. 
Moreover, we outline how our approach in combination with our software helped in a recent joint publication~\cite{cvetkovic2021algebraic} to find and automatically prove new results in the field of generalised inverses.
Finally, we sketch how the framework allows to formulate and automatically prove well-known statements in abelian categories that typically require tedious diagram chasing proofs.


\section{Algebraic preliminaries}
\label{sec:algebraic-preliminaries}

\subsection{Preadditive semicategories} 

As mentioned before, linear operators appear in various forms (ring elements, matrices, homomorphisms, etc.).
To treat all these different contexts uniformly, we require an algebraic structure that generalises across these settings.
To this end, we use concepts from category theory, in particular, that of \emph{preadditive semicategories}, whose basics we recall in this section.


Category theory is a powerful and abstract framework for studying mathematical structures and relationships between them.
For a textbook exposition of category theory, we refer, for example, to~\cite{Mac13}.
At its core, category theory explores the notion of a category, or more generally, that of a semicategory.

\begin{definition}
\label{def:semicategory}
A \emph{(locally small) semicategory} $\mathcal{S}$ consists of
\begin{itemize}
  \item a class $\Ob(\mathcal{S})$ of \emph{objects};
  \item for every two objects $U,V\in\Ob(\mathcal{S})$, a set $\Mor(U,V)$ of \emph{morphisms} from $U$ to $V$;
  for $A \in \Mor(U,V)$, we also write $A\colon U \to V$; the objects $U$ and $V$ are referred to as the \emph{source} and \emph{target} of $A$;
  \item for every three objects $U,V,W \in \Ob(\mathcal{S})$, a binary operation $\circ \colon \Mor(V,W) \times \Mor(U,V) \to \Mor(U,W)$ called 
  \emph{composition of morphisms}, which is associative, that is, $A \circ (B \circ C) = (A \circ B) \circ C$ for all $A \colon V \to W$, $B \colon U \to V$, $C \colon T \to U$;
\end{itemize}
A semicategory $\mathcal{S}$ is called \emph{preadditive} if every set $\Mor(U,V)$ is equipped with a binary operation $+$, turning it into an abelian group, such that composition of morphisms is bilinear, that is,
	\begin{align*}
		A \circ (B + C) = (A \circ B) + (A \circ C) \quad \text{ and } \quad (A + B) \circ C = (A \circ C) + (B \circ C).
	\end{align*}
The neutral element of the abelian group $\Mor(U,V)$ is called the \emph{zero morphism} from $U$ to $V$ and denoted by $0_{U,V}$.
\end{definition}

\begin{remark}
In a preadditive semicategory, we have
$(-A) \circ B = A \circ (-B) = -(A\circ B)$ and
$A \circ 0_{U,V} = 0_{U,W} = 0_{V,W} \circ B$
 for all $A \colon V \to W$ and $B \colon U \to V$.
\end{remark}

A semicategory can be thought of as a collection of objects, linked by arrows (the morphisms) that can be composed associatively. 
We note that the words \emph{object} and \emph{morphism} do not imply anything about the nature of these things.
Objects can be anything from numbers and sets to more complex structures like groups, vector spaces, or topological spaces, and similarly for morphisms.
However, while the terminology itself does not prescribe specific types, one can intuitively think of objects as sets and of morphisms as maps between these sets. 

The property of being locally small refers to the fact that the objects $\Ob(\mathcal{S})$ can form a proper class, while, for all objects $U,V \in \Ob(\mathcal{S})$, the morphisms $\Mor(U,V)$ constitute a set.
Preadditive semicategories have the additional property that morphisms with the same source and target can be added, yielding an abelian group structure that is compatible with the composition of morphisms.
For further information, see, for example,~\cite[Sec.~2]{garraway2005sheaves} or ~\cite[App.~B]{tilson}, where semicategories are also called semigroupoids.

Preadditive semicategories provide a natural and very general environment for studying linear operators, prescribing only linearity as a structural constraint. 
In particular, statements about linear operators across different settings can be handled uniformly by considering them as statements about morphisms in preadditive semicategories. 
In the following, we provide several examples of common settings that fall within this approach.

\begin{example}
In the following, $R$ denotes an arbitrary ring (not necessarily with $1$).
\begin{enumerate}
  \item The ring $R$ can be identified as a preadditive semicategory that consists of a single object, 
  say $\mathcal{R}$ for some (arbitrary) symbol $\mathcal{R}$, and morphism set $\Mor(\mathcal{R},\mathcal{R}) = R$.
  The abelian group operation on $\Mor(\mathcal{R},\mathcal{R})$ is given by the addition in $R$ and the composition of morphisms by the multiplication in $R$.
  Thus, in this setting, the morphisms are precisely the ring elements.

  \item The set $\textbf{Mat}(R)$ of matrices with entries in $R$ can be considered as a preadditive semicategory by taking as objects 
 the natural numbers $\N$ and letting $\Mor(n,m) = R^{m \times n}$ for all $n,m \in \N$, equipped with matrix addition as the abelian group operation.
  Composition is given by matrix multiplication.
  Here, morphisms correspond to matrices with entries in $R$.
  
  \item The category $R$-\textbf{Mod} of left modules over $R$ is a preadditive semicategory.
  Here, objects are left $R$-modules and morphisms are module homomorphisms between left $R$\nobreakdash-modules.  
  As a special case, also $K$-\textbf{Vect}, the category of vector spaces over a field $K$ with $K$-linear maps as morphisms, is a preadditive semicategory.
  In both cases, composition of morphisms is given by function composition and the group operation on the morphism sets is given by pointwise addition.
  Note that the objects in these categories form proper classes and not sets.
  
  \item For $K \in \{\mathbb{R}, \mathbb{C}\}$, the category \textbf{Hilb}$_{K}$, whose objects are Hilbert spaces over $K$ and whose morphisms are the bounded $K$-linear maps between these Hilbert spaces, is a preadditive semicategory.
 The composition of morphisms and the abelian group operation are again given by function composition and by pointwise addition, respectively.
    
  \item Another example of a preadditive semicategory is the category \textbf{Ab} of abelian groups, which has all abelian groups as objects and group homomorphisms as morphisms. 
It is equipped with function composition as composition operation and pointwise addition as the group operation on the morphism sets.
  
  \item Since a semicategory is just a category without identity morphisms, every preadditive category is a preadditive semicategory. 
  Thus, so are, in particular, abelian categories.
\end{enumerate}
\end{example}

\subsection{Noncommutative polynomials}
\label{sec:free-algebra}

The central algebraic objects that we consider in this work are noncommutative polynomials in a free algebra.
In this section, we recall their construction and main properties.

For a set $X = \{x_1,\dots,x_k\}$, the free monoid over $X$ is the set $\<X>$ of all finite words over the alphabet $X$, including the empty word $1$, equipped with concatenation of words as its monoid operation.
Furthermore, the free (unitary, associative) algebra $\Z\<X>$ on $X$ over the ring of integers $\Z$ is the set of finite formal sums
\begin{align*}
	\Z\<X> = \left\{ \sum_{w \in \<X>} c_w w \;\middle|\; c_w \in \Z\text{ and }c_w = 0\text{ for almost all }w\right\}
\end{align*}
together with the addition and multiplication
\begin{align*}
	 \sum_{w \in \<X>} c_w w +  \sum_{w \in \<X>} d_w w &= \sum_{w \in \<X>} (c_{w} + d_{w}) w \\
	\left(\sum_{u \in \<X>} c_u u \right) \cdot \left(\sum_{v \in \<X>} d_v v\right) &= \sum_{w \in \<X>} \sum_{uv = w} (c_{u} d_{v}) w.
\end{align*}

\begin{remark}
If $X = \{x\}$ is a singleton, then $\<X> = \{ x^{n} \mid n \in \N\}$ and $\Z\<X> = \Z[x]$ is simply the ring of univariate (commutative) polynomials.
 If $|X| > 1$, then $\Z\<X>$ is a noncommutative ring.
\end{remark}

\begin{remark}
More generally, the free algebra $R\<X>$ could be defined over an arbitrary commutative coefficient ring $R$.
For our application, however, only the case $R = \Z$ will be relevant, and hence, we focus on this special instance.
\end{remark}

We consider the elements in $\Z\<X>$ as \emph{noncommutative polynomials} with coefficients in $\Z$ and monomials in $\<X>$.
Note that the indeterminates in $\Z\<X>$ do not commute with each other, but they still commute with coefficients.

\begin{example}
For $f_1 = xy + x$, $f_2 = xy - 2y \in  \Z\< x,y>$, we have
\begin{align*}
f_1 + f_2 &= 2xy + x - 2y,\\
 f_1f_2 &= (xy + x)(xy - 2y) = xyxy + xxy - 2xyy - 2xy,\\
 f_2f_1 &= (xy-2y)(xy+x) = xyxy + xyx - 2yxy -2yx.
 \end{align*}
 Note that $f_1f_2 \neq f_2f_1$.
\end{example}

For a set of polynomials $F \subseteq \Z\<X>$, we denote by $(F)$ the \emph{(two-sided) ideal} generated by $F$, that is,
\[
	(F) = \left\{ \sum_{i=1}^d a_i f_i b_i \;\middle|\; f_i \in F, a_i,b_i \in \Z\<X>, d \in \N\right\}.
\]
As a special case, we obtain $(\emptyset) = \{0\}$.
An ideal $I \subseteq \Z\<X>$ is said to be \emph{finitely generated} if there exists a finite set $F \subseteq \Z\<X>$ such that $I = (F)$.
We note that $\Z\<X>$ is not Noetherian if $|X| > 1$, meaning that there exist ideals in $\Z\<X>$ which are not finitely generated.

Given (commutative or noncommutative) polynomials $f,f_{1},\dots,f_{r}$, the \emph{ideal membership problem} asks if $f$ belongs to the ideal 
generated by $f_{1},\dots,f_{r}$.
In the setting of commutative polynomials, the theory of Gr\"obner bases allows to decide this problem~\cite{BW93}, using, for example, Buchberger's algorithm~\cite{Buc65}.
We refer to~\cite{AL94,Cox} for textbook expositions and further details.

The theory of Gr\"obner bases also exists for noncommutative polynomials in the free algebra.
A noncommutative version of Buchberger's algorithm in the free algebra was first developed over coefficient fields~\cite{Mor85},
and later generalised to more general coefficient domains, and in particular to $\Z\<X>$~\cite{mikhalev1998standard,LMA20}.
In contrast to the commutative case, however, ideal membership of noncommutative polynomials is only semi-decidable.
In particular, ideals in the free algebra need not admit a finite Gr\"obner basis and the noncommutative Buchberger algorithm is, in fact, only an enumeration procedure.
This is a consequence of the undecidability of the word problem for semigroups~\cite[Thm.~4.5]{Dav58}, which can be reduced to the ideal membership 
problem in $\Z\<X>$, see, for example,~\cite[Rem.~2.2.12]{Xiu12}.

More precisely, \emph{verifying} membership of a polynomial in a finitely generated ideal is always possible.
For example, the noncommutative analogue of Buchberger's algorithm and the concept of polynomial reduction can be used as such a semi\nobreakdash-decision procedure.
We note that there exist implementations of these Gr\"obner bases methods in different computer algebra systems (see, for example,~\cite{Hof20} or~\cite{letterplace} and references therein), which allow to automatically verify ideal membership of noncommutative polynomials.
However, \emph{disproving} ideal membership in  the free algebra is not possible in general.

We summarise the previous discussion in the following proposition.

\begin{proposition}\label{prop ideal membership}
There exists a semi-decision procedure for verifying ideal membership in finitely generated ideals in $\Z\<X>$.
\end{proposition}

 
\section{An informal introduction to the framework}
\label{sec:uniqueness}

In this section, we give an informal and accessible introduction to our framework.
In particular, we illustrate, using the uniqueness of the Moore-Penrose inverse as an example, how identities of linear operators can be translated into noncommutative polynomials and how new identities can be derived from computations with these polynomials.
As mentioned previously, we treat linear operators as morphisms in preadditive semicategories, and thus, in the following, we shall only speak about morphisms. 

The Moore-Penrose inverse, originally described by E.~H.~Moore~\cite{moore1920reciprocal} and later rediscovered by Roger Penrose~\cite{Pen55}, generalises the notion of the inverse of a nonsingular square matrix to all, including rectangular, matrices.
More generally, the Moore-Penrose inverse can be defined for arbitrary morphisms in involutive preadditive semicategories.
A semicategory $\mathcal{S}$ is called \emph{involutive} if it is equipped with a map $\ast$, called \emph{involution}, on the morphisms of $\mathcal{S}$ that sends every $A \colon U \to V$ to some morphism $A^{\ast}\colon V \to U$ and satisfies 
\begin{align}\label{eq:involution}
(A^{\ast})^{\ast} = A &&\text{and}&& (AB)^{\ast} = B^{\ast}A^{\ast}.
\end{align}

\begin{definition}
Let $A\colon U \to V$ be a morphism in an involutive preadditive semicategory $\mathcal{S}$.
A morphism $B \colon V \to U$ is called a \emph{Moore-Penrose inverse} of $A$ if it satisfies the four \emph{Penrose identities}
\begin{align}\label{eq:moore-penrose}
    A B A = A, && B A B = B, && (AB)^* = AB, && (BA)^{*} = B A.
\end{align}
\end{definition}

A straightforward computation involving the Penrose identities, along with their adjoint versions that follow by applying the involution $\ast$ to both sides of an identity, shows that the Moore-Penrose inverse of $A$ is unique, if it exists.
More precisely, if $B$ and $C$ both satisfy~\eqref{eq:moore-penrose}, then
\begin{align}\label{eq:proof-uniqueness}
\begin{aligned}
  B &= BAB = BACAB = BACB^*A^* = BC^*A^*B^*A^* \\
    &= BC^*A^* = BAC = A^*B^*C = A^*C^*A^*B^*C \\
    &= A^*C^*BAC = CABAC = CAC = C.
\end{aligned}
\end{align}

In the following, we describe how the uniqueness of the Moore-Penrose inverse can be proven by a computation with noncommutative polynomials.

Intuitively, an identity $B = C$ of morphisms in a preadditive semicategory can be identified with the noncommutative polynomial $b - c \in \Z\<b,c>$.
More generally, identities of composite morphisms can be translated into noncommutative polynomials by introducing a noncommutative indeterminate for each basic nonzero morphism, and by uniformly replacing each morphism by the respective indeterminate in the difference of the left- and right-hand side of each identity.
Potentially present zero morphisms are replaced by the zero in $\Z\<X>$.

For example, by introducing indeterminates $a, b, a^*, b^*$ representing the morphisms $A,B,A^{\ast}, B^{\ast}$, 
the Penrose identities~\eqref{eq:moore-penrose} can be translated into the following polynomials in $\Z\<a,b,a^{*}, b^{*}>$:
\begin{align}\label{eq:moore-penrose-polynomials}
  &aba - a, && bab - b, && b^* a^* - ab, && a^* b^* - ba.
\end{align}
Note that, here, we used the properties~\eqref{eq:involution} of the involution to simplify the Penrose identities before translating them into polynomials.


The following theorem is a variant of the main result in~\cite{RRH21} and says that, using the translation described above,
a computation like~\eqref{eq:proof-uniqueness} can be inferred from a particular polynomial ideal membership.
In the following, we associate to each identity of morphisms $S = T$ the noncommutative polynomial $s - t$ through the described translation.

\begin{theorem}[{\cite[Thm.~32]{RRH21}}]\label{cor:quasi-identities}
An identity $P = Q$ of morphisms in a preadditive semicategory follows from other identities $S_1 =  T_1,\dots, S_n = T_n$
if 
the noncommutative polynomial $p - q$ lies in the ideal generated by $s_1 - t_1,\dots, s_n - t_n$ in the free algebra $\Z\<X>$.
\end{theorem}


In this work, we will generalise Theorem~\ref{cor:quasi-identities} in several ways, extending its scope first to arbitrary propositional statements about identities of morphisms (Theorem~\ref{thm nullstellensatz}) and later to first-order statements that can additionally involve quantifiers and function symbols. 
Ultimately, we will obtain a semi-decision procedure for verifying first-order statements about identities of morphisms (Procedure~\ref{algo semantic consequence}).


To conclude our introductory example, the theorem stated above suffices.
By Theorem~\ref{cor:quasi-identities}, the uniqueness of the Moore-Penrose inverse can be proven
by showing that the polynomial $b - c$ lies in the ideal generated by the polynomials encoding that $B$ and $C$ satisfy the Penrose identities for $A$, that is,
$b - c \in (F) \subseteq \Z\<a,b,c,a^{*},b^{*},c^{*}>$, where
\begin{align*}
	F = \{ &aba - a, bab - b, b^* a^* - ab, a^* b^* - ba, a^{*} b^{*} a^{*} - a^{*}, b^{*} a^{*} b^{*} - b^{*},\\
	&aca - a, cac - c, c^* a^* - ac, a^* c^* - ca, a^{*} c^{*} a^{*} - a^{*}, c^{*} a^{*} c^{*} - c^{*}\}.
\end{align*}

\begin{remark}\label{remark:involution}
For correctly representing the involution $*$ in the free algebra, we introduce, analogous to~\eqref{eq:moore-penrose-polynomials}, 
a separate indeterminate for the involution of each basic morphism.
Subsequently, we rewrite all expressions using the identities~\eqref{eq:involution} before translating them into polynomials

Moreover, when an identity $S = T$ is valid, the corresponding adjoint identity $S^* = T^*$ also holds, 
necessitating also the translation of these additional identities into polynomials.
Therefore, to express that $B$ is a Moore-Penrose inverse of $A$ with polynomials, 
we have to augment~\eqref{eq:moore-penrose-polynomials} with the polynomials corresponding to the adjoint identities.
Given that the last two Penrose identities are self-adjoint, this yields the two additional elements $a^* b^* a^* - a^*$ and $b^* a^* b^* - b^*$.

We note that this treatment of the involution deviates from how Procedure~\ref{algo semantic consequence} is phrased.
Instead, this corresponds to applying one of the techniques discussed in Section~\ref{sec computational aspects},
which leads to a simpler and more efficient computation.
\end{remark}

The ideal membership of $b-c$ can be verified automatically using dedicated computer algebra software that implements noncommutative Gr\"obner bases methods, 
like our computer algebra package \texttt{operator\_gb}.
Notably, employing such software allows the computation of a representation of the ideal element as a linear combination of the generators, which
serves as a verifiable certificate for the ideal membership.
For our example, we can compute the following representation:
\begin{align}\label{eq:proof-uniqueness-polynomials}
\begin{aligned}
  b - c =\;&-(bab - b) \;-\; b(aca - a)b \;-\; bac(b^* a^* - ab) \;-\; b(c^*a^* - ac)b^*a^* \\
  +\;&bc^*(a^*b^*a^* - a^*) \;+\; b(c^*a^* - ac) \;-\; (a^*b^* - ba)c  \;-\; (a^*c^*a^* - a^*)b^*c\\
 +\;& a^*c^*(a^*b^* - ba)c \;+\; (a^*c^* - ca)bac \;+\; c(aba - a)c \;+\; (cac - c).
\end{aligned}
\end{align}
This verifies the ideal membership of $b-c$ in the ideal generated by the polynomials in the set $F$, and thus, based on Theorem~\ref{cor:quasi-identities}, we arrive at the following conclusion.

\begin{proposition}
The Moore-Penrose inverse of a morphism in an involutive preadditive semicategory is unique, if it exists.
\end{proposition}

By specialising the general setting of preadditive semicategories to more concrete contexts, we immediately obtain the uniqueness statement of the Moore-Penrose inverse, for example, for elements in arbitrary rings with involution (such as $C^*$\nobreakdash-algebras), for complex matrices, or for bounded linear operators on Hilbert spaces.

Thus far, we have treated the notion of an \emph{operator statement} and its correctness rather intuitively. 
However, in order to develop our framework, we require precise definitions of these concepts.
Therefore, in the following section, we use many-sorted first-order logic to specify what we mean by operator statements and what it means for such a statement to be \emph{universally true}.


\section{Modelling operator statements using many-sorted logic}
\label{sec:modelling-operator-statements}

In this section, we adapt the general concept of many-sorted logic to model statements about linear operators and, more generally, about morphisms in preadditive semicategories.
For the reader's convenience, we provide an introduction to many-sorted logic in~\ref{sec:many-sorted-logic}, 
where also definitions of the basic terminology used in the following can be found.

In particular, we use variables and constant symbols to model morphisms (i.e., linear operators), and we exploit the concept of sorts to represent the sources and targets of these morphisms (i.e., the domains and codomains of the operators).
We also define what it means for an operator statement to be \emph{universally true}, and we introduce a special class of operator statements, called \emph{arithmetic operator statements}.
This restricted class allows for a direct translation into algebraic assertions about noncommutative polynomials and will be the foundation for treating arbitrary operator statements. 

\subsection{Operator statements and universal truth}
\label{sec:operator-statements-and-universal-truth}

For defining sorts, we fix a denumerable set $\ObS = \{v_1,v_2,\dots \}$ of \emph{object symbols}, and we consider the pairs $(u,v) \in \ObS \times \ObS$ as sorts,
that is, we specialise $\Sort = \ObS \times \ObS$.
Note that our sorts are pairs of symbols, a deliberate choice necessary for adequately representing both the source and target of the morphisms (i.e., the domain and codomain of the operators) we aim to model.

Furthermore, for each sort $(u,v)$, we identify a distinguished \emph{zero constant symbol} $0_{u,v} \in \Con$, representing a dedicated zero morphism.
We collect them in the set $\Zero = \{ 0_{u,v} \mid u,v \in \ObS\}$.
Similarly, we also distinguish certain special function symbols and consider the set of \emph{arithmetic function symbols} $\Arith = \bigcup_{u,v,w \in \ObS} \{ -_{u,v},\, +_{u,v},\, \cdot_{u,v,w}\} \subseteq \Fun$.
These function symbols shall model the arithmetic operations of additive inversion, addition, and composition of morphisms respectively.

A signature lists and describes the non-logical symbols of a logic that are relevant in a particular context, see Definition~\ref{def:signature-logic} in the appendix.
For our setting, we consider signatures of the form $\Sigma = (\bm{O} \times \bm{O}, \bm{C}, \bm{F}, \sigma)$, where
\begin{enumerate}
	\item $\bm{O} \subseteq \ObS$ is a set of object symbols;
	\item $\bm{C} \subseteq \Con$ such that $0_{u,v} \in C$ for all $u,v \in \bm{O}$;
	\item $\bm{F} \subseteq \Fun$ such that $-_{u,v},\, +_{u,v},\, \cdot_{u,v,w} \in F$ for all $u,v,w \in \bm{O}$;
	\item the sort function $\sigma$ satisfies the following conditions:
	\begin{enumerate}
	\item $\sigma(0_{u,v}) = (u,v)$ for all $0_{u,v} \in \bm{C}$;
	\item $\sigma( -_{u,v}) = (u,v) \to (u,v)$ for all $-_{u,v} \in \bm{F}$;
	\item $\sigma(+_{u,v}) = (u,v) \times (u,v) \to (u,v)$ for all $+_{u,v} \in \bm{F}$;
	\item $\sigma(\cdot_{u,v,w}) = (v,w) \times (u,v) \to (u,w)$ for all $\cdot_{u,v,w} \in \bm{F}$;
	\end{enumerate}
\end{enumerate}

Above, we used the notation $\sigma(f) = s_1 \times \dots \times s_n \to s$ to represent the sort $\sigma(f) = (s_1,\dots,s_n,s)$ of a function symbol $f$,
see also Remark~\ref{remark:function-symbols}.

We give some remarks on how signatures now specialise compared to the general Definition~\ref{def:signature-logic}.
\begin{itemize}
	\item As also noted before Definition~\ref{def:signature-logic}, in our setting, signatures do not contain any predicate symbols. 
	The only predicate that we will work with is equality $\approx$, which will be interpreted as identity.
	\item The conditions on the sets $\bm{C}$ and $\bm{F}$ of relevant constant and function symbols ensure that they contain at least the zero constants
	and the arithmetic function symbols.
	\item The conditions on the sort function $\sigma$ ensure that the dedicated zero constants and the arithmetic function symbols have the sorts one would expect, mirroring the arithmetic operations of morphisms (cf.~Definition~\ref{def:semicategory}).
\end{itemize}

In the following, we assume that all signatures we consider are of the form described above.
Using this special kind of signature, we can now rigorously define \emph{operator statements}.
To this end, we fix a signature $\Sigma = (\bm{O} \times \bm{O},\bm{C},\bm{F},\sigma)$ for the rest of this section.

\begin{definition}
Any formula $\varphi \in \Formulas(\Sigma)$ is called an \emph{operator statement}.
\end{definition}

As a special class of operator statements, we can formulate the \emph{axioms of preadditive semicategories}. 
We denote the set of these axioms by $\Axioms$, and to define it, we introduce the auxiliary sets $\Axioms_{t,u,v,w}$, with $t,u,v,w \in \bm{O}$, containing the sentences listed below.
To make the sorts of the variables used better visible, we write $x^{(u,v)}$ in the prefix if $x$ is a variable of sort $(u,v)$.
Furthermore, to simplify the notation, we use the same symbols $-$, $+$, and~$\cdot$ for all arithmetic function symbols $-_{u,v}$, $+_{u,v}$, and $\cdot_{u,v,w}$.
We also use the usual infix notation for $+$ and $\cdot$, and we assume the natural precedence of these operations.

\begin{equation*}
\Axioms_{t,u,v,w} \;=\; \left\lbrace\qquad\begin{aligned}
    	 &\forall x^{(v,w)}, y^{(u,v)}, z^{(t,u)} &&:&&  x \cdot (y \cdot z) \approx (x \cdot y) \cdot z, \\
	 &\forall x^{(u,v)}, y^{(u,v)}, z^{(u,v)} &&:&& x + (y + z) \approx (x + y) + z, \\
	 &\forall x^{(u,v)} &&:&& x + 0_{u,v} \approx x, \\
	 &\forall x^{(u,v)} &&:&& x + (-x) \approx 0_{u,v}, \\
	 &\forall x^{(u,v)}, y^{(u,v)} &&:&& x + y \approx y + x, \\
	 &\forall x^{(v,w)}, y^{(u,v)}, z^{(u,v)} &&:&&  x \cdot (y + z) \approx x \cdot y + x \cdot z, \qquad \\
	 &\forall x^{(v,w)}, y^{(v,w)}, z^{(u,v)} &&:&&  (x + y) \cdot z \approx x \cdot z + y \cdot z \qquad \\
  \end{aligned}\right\rbrace
\end{equation*}

Note that the sorts of the used arithmetic function symbols are determined implicitly by the sorts of the variables.
For example, adding the sorts of the function symbols explicitly to the last axiom listed above yields
\begin{align*}
		 \forall x^{(v,w)}, y^{(v,w)}, z^{(u,v)} \;:\;  (x +_{v,w} y) \cdot_{u,v,w} z \,\approx\, x \cdot_{u,v,w} z +_{u,w} y \cdot_{u,v,w} z.
\end{align*}

\begin{definition}
Let $\Sigma = (\bm{O} \times \bm{O},\bm{C},\bm{F},\sigma)$ be a signature and $\Axioms_{t,u,v,w}$ be as above.
The set $\Axioms$ of \emph{axioms of preadditive semicategories in signature $\Sigma$} is given by
\[
	\Axioms \coloneqq \bigcup_{t,u,v,w \in \bm{O}} \Axioms_{t,u,v,w}.
\]
\end{definition}

Note that $\Axioms$ is finite if and only if $\bm{O}$ is.
If $\Sigma$ is clear from the context, we refer to $\Axioms$ simply as the set of axioms of preadditive semicategories and omit the dependency on $\Sigma$.
We make some remarks about the nature of models $\mathcal{I} = (A,\assign)$ of the axioms $\Axioms$ in signature $\Sigma = (\bm{O} \times \bm{O},\bm{C},\bm{F},\sigma)$.
For definitions of the logical notions used below, we refer to~\ref{sec:many-sorted-logic}.

First, since all formulas in $\Axioms$ are sentences, that is, they do not contain any free variables, the models of $\Axioms$ do not depend on the variable assignment $\assign$ but only on the structure $A$.
If $\assign$ and $\mathfrak{b}$ are different assignments, then $(A,\assign)$ is a model of $\Axioms$ if and only if $(A,\mathfrak{b})$ is.

Moreover, any structure $A = ((A_{u,v})_{u,v\in \bm{O}}, (c^{A})_{c \in \bm{C}}, (f^{A})_{f \in \bm{F}}) $ that is part of a model $(A,\assign)$ of $\Axioms$ can be considered as a preadditive semicategory $\mathcal{S}$ with objects $\Ob(\mathcal{S}) = \bm{O}$ and morphism sets $\Mor(u,v) = A_{u,v}$.
The arithmetic operations in $\mathcal{S}$, that is, composition of morphisms, addition, and additive inversion, are given by the interpretations of the arithmetic function symbols.
Furthermore, the axioms $\Axioms$ ensure that each set $A_{u,v}$ is indeed an abelian group with neutral element given by the interpretation of the constant symbol $0_{u,v}$ and that the composition of morphisms is bilinear.

Conversely, given any preadditive semicategory $\mathcal{S}$ with $\Ob(\mathcal{S}) \neq \emptyset$, a model $(A,\assign)$ of $\Axioms$ can be constructed as follows:
For the structure $A$, take any mapping of the object symbols $\bm{O}$ to objects in $\mathcal{S}$, say $\pi \colon \bm{O} \to \Ob(\mathcal{S})$, and choose as domains $A_{u,v} = \Mor(\pi(u), \pi(v))$ for all $u,v \in \bm{O}$.
Furthermore, take the zero morphisms $0_{\pi(u),\pi(v)}$ as the interpretations of the zero constant symbols $0_{u,v}$ and the arithmetic operations in $\mathcal{S}$ as the interpretations of the arithmetic function symbols.  
Then, by construction, for any variable assignment $\assign$, the pair $(A,\assign)$ forms a model of $\Axioms$.

We are interested in the first-order theory axiomatised by the set $\Axioms$, that is, all statements that follow from the axioms of preadditive semicategories.
Since the axioms $\Axioms$ only prescribe basic linearity assumptions, we can consider any such statement as a universally true operator statement.
Formally, we arrive at the following definition.

\begin{definition}\label{def:universal-truth}
An operator statement $\varphi \in \Formulas(\Sigma)$ is \emph{universally true} if $\Axioms \models \varphi$.
\end{definition}

Based on G\"odel's completeness theorem (Theorem~\ref{thm gödel}), the universal truth of an operator statement $\varphi$ can be proven by a formal computation.
For instance, using the sequent calculus LK$^=$ presented in~\ref{sec:formal-comp}, this process requires deriving the sequent $\Axioms \vdash \varphi$ from the axioms of the calculus.
However, such formal computations present several drawbacks.
Notably, the linearity axioms $\Axioms$ have to be handled separately, and one has to ensure that all expressions formed during such a formal computation respect the restrictions imposed by the sorts.
In the following sections, we propose an alternative approach, replacing formal computations by computations with noncommutative polynomials, 
which naturally include linearity and, as will be shown, for which restrictions induced by sorts can be neglected.

In Section~\ref{sec:uniqueness}, we have already seen (intuitively) how very elementary operator statements can be treated.
Specifically, the statements addressed by Theorem~\ref{cor:quasi-identities} are restricted to basic implications and fall into a class of operator statements that we call \emph{arithmetic} operator statements.
In the following, we formally define this class.

\subsection{Arithmetic operator statements}
\label{sec:def-arithmetic}

\begin{definition}
An operator statement $\varphi \in \Formulas(\Sigma)$ is called \emph{arithmetic} if it is quantifier-free and all function symbols appearing in $\varphi$ are arithmetic function symbols.
\end{definition}

\begin{example}
Consider a signature $\Sigma = (\bm{O} \times \bm{O},\bm{C},\bm{F},\sigma)$ with $\bm{O} = \{u,v\}$ and non-arithmetic function symbols $f,g \in \bm{F}$ such that 
\[
	\sigma(f) = (u,v) \to (u,v), \qquad \sigma(g) = (v,u) \times (u,v) \to (u,u).
\]
Furthermore, let $x,y \in \Var$ with $\sigma(x) = (u,v)$ and $\sigma(y) = (v,u)$. 

The quantifier-free operator statement $x\cdot y \cdot x + f(x) \approx x \vee y \cdot x \approx g(y,x)$ is not arithmetic as it contains the function symbols $f$ and $g$. 
An example of an arithmetic operator statement would be $x\cdot y\cdot x \approx x \wedge y\cdot x\cdot y \not\approx y$.
\end{example}

All terms appearing in an arithmetic operator statement are essentially polynomial expressions in some basic symbols (the variables and constant symbols).
Thus, it is straightforward to translate terms appearing in such formulas into noncommutative polynomials, simply by 
translating variables and constant symbols into indeterminates and the arithmetic function symbols into the arithmetic operations in $\Z\<X>$.
Note that, since any formula is a finite string of symbols, it is clear that the polynomials obtained in this way consist of finitely many indeterminates.
We describe this translation formally in Section~\ref{sec:terms-to-polynomials}.

In Section~\ref{sec:arithmetic-statements}, we also give with our main result (Theorem~\ref{thm nullstellensatz})
an algebraic characterisation of universal truth of arithmetic operator statements via ideal membership of noncommutative polynomials.
However, before doing this, we first address in the following section how the treatment of arbitrary operator statements can be systematically reduced to dealing with arithmetic operator statements.


\section{Reduction to arithmetic operator statements}
\label{sec:herbrand-ackermann}

We have already anticipated that arithmetic operator statements allow for a natural translation into algebraic statements (see Section~\ref{sec:arithmetic-statements} for further details). 
However, arbitrary operator statements, which, in contrast to arithmetic ones, can involve arbitrary function symbols as well as quantifiers, cannot be handled directly.
Therefore, to be able to also treat formulas involving these concepts, we recall two important notions from first-order logic and automated theorem proving, 
namely Herbrand’s theorem~\cite{Her30,buss1994herbrand} and Ackermann’s reduction~\cite{Ack54}.
By combining these methods, we are able to reduce the treatment of arbitrary operator statements to that of arithmetic operator statements (see Section~\ref{sec:all-operator-statements}).

We note that, in the following, we rather speak of \emph{formulas} than \emph{operator statements} in order to emphasise the fact that the discussed concepts work for arbitrary first-order formulas, and that no special adaptations are needed for our application of proving operator statements. 
Therefore, we also work with general signatures $\Sigma = (\bm{S},\bm{C},\bm{F},\sigma)$ as defined in Definition~\ref{def:signature-logic}, where $\bm{S}$ is an arbitrary set of sorts. 

\subsection{Herbrand's theorem and Herbrandisation}\label{sec:herbrand}

Herbrand's theorem, a fundamental result in logic, provides a systematic way to eliminate quantifiers from first-order formulas, and in this way, 
essentially allows a reduction from first-order logic to propositional logic.
In particular, it says that certain semantic consequences can be verified by a formal computation involving only ground instances of the original formulas.

Before we proceed to state Herbrand's theorem, we first recall some normal forms of formulas as well as the notion of \emph{Herbrand expansion}.
For the following, we fix a signature $\Sigma = (\bm{S},\bm{C},\bm{F},\sigma)$.

A sentence of the form $\varphi = Q_1 x_1 \dots Q_n x_n : \varphi^*$ with $\varphi^*$ quantifier-free and $Q_i \in \{\exists, \forall\}$ is in \emph{prenex normal form} with \emph{prefix} $Q_1 x_1 \dots Q_n x_n$.
We note that any formula can be transformed into a logically equivalent formula in prenex normal form.
If the prefix consists only of existential (resp.~universal) quantifiers, i.e., $Q_i =  \exists$ (resp.~$Q_i = \forall$) for all $i$, then $\varphi$ is in \emph{Herbrand normal form} (resp.~\emph{Skolem normal form}).

The \emph{Herbrand expansion} $H(\varphi)$ of a sentence $\varphi = Q_1 x_1 \dots Q_n x_n \colon \varphi^*$ in prenex normal form is the set of all ground instances of $\varphi$, that is, if $n > 0$ then
\[
	H(\varphi) \coloneqq \left\{ \varphi^*[x_1 \mapsto t_1, \dots, x_n \mapsto t_n] \mid t_i  \in \Ground(\Sigma), \sigma(t_i) = \sigma(x_i) \right\},
\]
and $H(\varphi) \coloneqq \{\varphi \}$ otherwise.

\begin{remark}
The Herbrand expansion depends on the signature $\Sigma$ and, in particular, on the sort function $\sigma$.
Different signatures and sort functions lead to different expansions.
\end{remark}

Note that, for the signatures we consider in our application, the set $H(\varphi)$ is either a singleton (in case $\varphi$ does not contain any variables, i.e., $n = 0$) or infinite.
This follows from the presence of the arithmetic function symbols in all our signatures which can be nested arbitrarily deep.
All formulas in $H(\varphi)$ are ground, meaning that they do not contain any variables. 
We extend the Herbrand expansion also to sets of sentences $\Phi \subseteq \Sent(\Sigma)$ in prenex normal form, by setting 
$H(\Phi) \coloneqq \bigcup_{\varphi \in \Phi} H(\varphi)$.
For any countable set $\Phi$, also $H(\Phi)$ is countable, which follows from the fact that $\Ground(\Sigma)$ is countable.

In the following, we state Herbrand's theorem.
We omit its proof as it can be found in any standard textbook about first-order logic.
See, for example,~\cite[Thm.~XI.1.4]{EFT94} for a formulation in classical unsorted first-order logic together with a corresponding proof for that setting.
We note that the proof in the unsorted case carries over in a straightforward way to the sorted setting.
Additionally,~\cite[Thm.~30]{NDFK10} provides an equivalent formulation of Herbrand's theorem proven for order-sorted logic.
Classical (unordered) many-sorted logic can be considered as a special case of order-sorted logic.

\begin{theorem}[Herbrand's theorem]\label{thm:herbrand}
Let $\Phi \subseteq \Sent(\Sigma)$ be a set of sentences in Skolem normal form and let $\varphi \in \Sent(\Sigma)$ be in Herbrand normal form.
The following are equivalent:
\begin{enumerate}
	\item $\Phi \models \varphi$.\label{item herbrand 1}
	\item There exist $\psi_1,\dots,\psi_k \in H(\Phi)$, $\varphi_1,\dots,\varphi_l \in H(\varphi)$ such that
	$\psi_1 \wedge \dots \wedge \psi_k \vdash \varphi_1 \vee \dots \vee \varphi_l$.\label{item herbrand 2}
\end{enumerate}	 
\end{theorem}


The axioms $\AA$ of preadditive semicategories form a set of sentences in Skolem normal form.
Hence, Theorem~\ref{thm:herbrand} can be used to reduce semantic consequences $\AA \models \varphi$ to syntactic consequences of ground sentences.
However, it might seem that the version of Herbrand's theorem stated above has limited applicability since it requires $\varphi$ to be in Herbrand normal form.
Fortunately, for any set $\Phi \subseteq \Sent(\Sigma)$ of axioms, every formula can be transformed into a $\Phi$-equivalid sentence in Herbrand normal form.
Recall that two formulas $\varphi$ and $\psi$ are called $\Phi$-equivalid if $\Phi \models \varphi$ if and only if $\Phi \models \psi$.

Starting from an arbitrary formula $\varphi$, first, any variables not in the scope of a quantifier are bound by a universal quantifier, that is,
if $\bm{x} = x_1,\dots,x_n$ are not in the scope of a quantifier, then $\varphi$ is transformed into the sentence $\forall \bm{x} : \varphi$.
The latter formula is called the \emph{universal closure} of $\varphi$.
By the semantics of the universal quantifier, $\varphi$ and $\forall \bm{x} : \varphi$ are $\Phi$-equivalid.

Next, $\forall \bm{x}:\varphi$ is transformed into prenex normal form.
After this process, the resulting sentence is almost in Herbrand normal form. 
However, the prefix of a formula in prenex normal form can still contain universal quantifiers.
The goal of \emph{Herbrandisation} is to remove these universal quantifiers in such a way that validity is preserved.
More precisely, given a sentence in prenex normal form, applying the following rules exhaustively is called \emph{Herbrandisation}.
\begin{align*}
	\forall y : \psi \quad &\rightsquigarrow \quad \psi[y \mapsto c] \\
	\exists x_1,\dots,x_n \forall y : \psi \quad &\rightsquigarrow \quad \exists x_1,\dots,x_n : \psi[y \mapsto f(x_1,\dots,x_n)]
\end{align*}
In these rules, $c$ is a new constant symbol with sort $\sigma(c) = \sigma(y)$ and $f$ is a new function symbol with sort $\sigma(f) = \sigma(x_1) \times \dots \times \sigma(x_n) \to \sigma(y)$.
Note that, whenever a Herbrandisation rule is applied, the signature $\Sigma = (\bm{S},\bm{C},\bm{F},\sigma)$ has to be extended, yielding a new signature.
In case of applying the first rule, $c$ has to be added to $\bm{C}$, and for the second rule, $f$ has to be added to $\bm{F}$.
In both cases, $\sigma$ has to be extended as well.

It is clear that Herbrandisation transforms well-formed sentences $\varphi$ into well-formed sentences in Herbrand normal form.
Additionally, it can be shown that the result of this transformation is also $\Phi$-equivalid to $\varphi$.
We refer to~\cite[Ch.~11]{Wa82}, where the dual concept of \emph{Skolemisation} is discussed, for further details.

Combining the steps described above, we obtain the following result.

\begin{proposition}\label{prop herbrandisation}
Let $\Phi \subseteq \Sent(\Sigma)$ be a set of sentences.
Every formula $\varphi \in \Formulas(\Sigma)$ can be transformed into a $\Phi$-equivalid sentence in Herbrand normal form.
\end{proposition}

In particular, every operator statement $\varphi$ can be transformed into a sentence $\varphi'$ in Herbrand normal form such that $\varphi$ is universally true if and only if $\varphi'$ is.
Then, using Herbrand's theorem, the treatment of $\varphi'$ can be reduced to that of finitely many ground instances of $\varphi'$, eliminating all quantifiers in this way.

\subsection{Ackermann's reduction}\label{sec:ackermann}

When translating operator statements into polynomial statements, it is natural to interpret the arithmetic function symbols as the arithmetic operations on polynomials.
However, other function symbols -- such as those introduced by Herbrandisation -- cannot be translated directly.
To treat such function symbols, we use a technique from the theory of uninterpreted functions with equality, called \emph{Ackermann's reduction}.

The goal of Ackermann's reduction is to remove function symbols from a formula in such a way that validity is preserved.
The reduction requires replacing function symbols with new constants and adding a set of constraints to enforce \emph{functional consistency}, 
which encodes that two instances of the same function symbol are equal if instantiated with equal arguments.

Ackermann's reduction is described formally in Algorithm~\ref{algo ackermann}.
We also refer to~\cite[Sec.~3.3.1]{KS16} for a textbook exposition of this procedure in the unsorted case and note that the addition of sorts is straightforward.

Algorithm~\ref{algo ackermann} is presented for the elimination of a single function symbol and follows closely~\cite[Alg.~3.3.1]{KS16}.
By repeated application, all unwanted function symbols can be removed from a given formula.
Note that this reduction introduces new constant symbols, which -- analogous to the Herbrandisation -- requires to extend the signature.
In the following, we fix a signature $\Sigma = (\bm{S},\bm{C},\bm{F},\sigma)$.

\begin{algorithm}
\renewcommand{\algorithmicrequire}{\textbf{Input:}}
\renewcommand{\algorithmicensure}{\textbf{Output:}}
\caption{Ackermann's reduction}\label{algo ackermann}
\begin{algorithmic}[1]
\Require{Quantifier-free formula $\varphi$ with $m$ different instances of function symbol $f$ of sort $\sigma(f) = s_{1} \times \dots \times s_{n} \to s$}
\Ensure{Quantifier-free formula $\varphi^{\text{Ack}(f)}$ without instances of $f$ s.t.\ $\varphi$ is valid iff $\varphi^{\text{Ack}(f)}$ is.}

\State{Assign indices to the $m$ instances of $f$ from subexpressions outwards.
Denote by $f_i$ the instance of $f$ with index $i$ and by $t^{(i)}_1,\dots,t^{(i)}_{n}$ the arguments of $f_i$.}
\State{Let $\acker$ be the function that maps every instance $f_i(t^{(i)}_1,\dots,t^{(i)}_{n})$ 
to a new constant symbol $c_i$ with sort $\sigma(c_i) = s$ and leaves all other terms unchanged.
In case of nested functions, the function $\acker$ is only applied to the outermost instance.}
\State{$\varphi^\text{flat} \leftarrow \acker(\varphi)$}
\State{Let $\varphi^\text{FC}$ denote the following formula
\[
	\bigwedge_{i=1}^{m-1} \bigwedge_{j=i+1}^m \left(\acker(t^{(i)}_1) \approx \acker(t^{(j)}_1) \wedge \dots \wedge \acker(t^{(i)}_n) \approx \acker(t^{(j)}_n) \right) \rightarrow c_i \approx c_j
\] 
}
\State{$\varphi^{\text{Ack}(f)} \leftarrow \left(\varphi^\text{FC} \rightarrow \varphi^\text{flat}\right)$}

\State{\textbf{return } $\varphi^{\text{Ack}(f)}$}
\end{algorithmic}
\end{algorithm}

\begin{proposition}\label{prop ackermann}
Algorithm~\ref{algo ackermann} is correct.
\end{proposition}

\begin{proof}
The resulting formula is clearly a quantifier-free formula without instances of $f$.
It remains to show that $\varphi$ is valid if and only if $\varphi^{\text{Ack}(f)}$ is, but this follows from the more general Proposition~\ref{prop ackermann cons} below by setting $\Phi = \emptyset$.
\end{proof}

\begin{example}
Consider the formula
\[
	\varphi = (x \not\approx y) \wedge f(x) \approx f(z) \wedge (x \not\approx f(x) \vee f(z) \not\approx f(f(y))).
\]
with variables $x,y,z$ of sort $s \in \bm{S}$ and function symbol $f$ with $\sigma(f) = s \to s$.
We apply Algorithm~\ref{algo ackermann} to remove $f$ from $\varphi$.

After assigning indices to the instances of $f$ (say, from left to right and, as required, inside-out),
we compute $\varphi^\text{flat}$ and $\varphi^\text{FC}$ accordingly:
\begin{align*}
	\varphi^\text{flat} =\; &(x \not\approx y) \wedge c_1 \approx c_2 \wedge (x \not\approx c_1 \vee c_2 \not\approx c_4),\\
	\varphi^\text{FC} =\; &(x \approx z \rightarrow c_1 \approx c_2) \;\wedge\; 
					(x \approx y \rightarrow c_1 \approx c_3) \\
			 \wedge\; &(x \approx c_{3} \rightarrow c_1 \approx c_4) \; \wedge\; 
			 		(z \approx y \rightarrow c_2 \approx c_3) \\
			    \wedge\; &(z \approx c_{3} \rightarrow c_2 \approx c_4) \; \wedge\; 
			    		(y \approx c_{3} \rightarrow c_3 \approx c_4),
\end{align*}
where $c_1,c_2,c_3,c_{4}$ are new constants of sort $s$ replacing $f(x), f(z), f(y), f(f(y))$ respectively.
Then $\varphi$ is valid if and only if $\varphi^{\text{Ack}(f)} = \varphi^\text{FC} \rightarrow \varphi^\text{flat}$ is, and in fact, one can show that $\varphi^{\text{Ack}(f)}$ is not valid, implying that neither is $\varphi$.
\end{example}

In our application, we are interested in semantic consequences $\Phi \models \varphi$.
The following result asserts that Ackermann's reduction also preserves this property for quantifier-free formulas $\varphi$
if the removed function symbol does not appear in $\Phi$.

\begin{proposition}\label{prop ackermann cons}
Let $\varphi \in \Formulas(\Sigma)$ be a quantifier-free formula containing a function symbol $f$.
Let $\Phi \subseteq \Formulas(\Sigma)$ be a set of formulas not containing $f$.
Then $\varphi$ is $\Phi$\nobreakdash-equivalid to $\varphi^{\text{Ack}(f)}$. 
\end{proposition}

\begin{proof}
We prove $\Phi \not\models \varphi$ if and only if $\Phi \not \models \varphi^{\text{Ack}(f)}$.
To this end, we consider all interpretations within the signature that contains both $f$ and the new constants $c_i$.
We also note that, as mentioned in Remark~\ref{remark:equality}, the equality predicate $\approx$ is always interpreted as identity,
which means that all our interpretations satisfy the axioms of equality (in particular, the congruence axioms).

For one implication, assume that $\Phi \not\models \varphi$.
Let $\II = (A,\assign)$ be a model of $\Phi$ such that $\II(\varphi) = \bot$.
Consider the interpretation $\II' = (A',\assign)$, where $A'$ is obtained from $A$ by setting $c_i^A = \II(f_i(t_1^{(i)},\dots,t^{(i)}_{n}))$, for all $i$, and leaving everything else unchanged.
Since $c_i$ are newly introduced constant symbols, they do not appear in $\Phi$.
As the interpretations of all other terms remain unchanged, $\II$ and $\II'$ agree on all terms in $\Phi$, and therefore, $\II'(\Phi) = \II(\Phi) = \top$. 
Furthermore, $\II'(\varphi^\text{flat}) = \II(\varphi) = \bot$ and $\II'(\varphi^\text{FC}) = \top$. 
Consequently $\II(\varphi^{\text{Ack}(f)}) = \bot$, showing that $\varphi^{\text{Ack}(f)}$ is not a semantic consequence of $\Phi$.

For the other implication, assume that $\Phi \not\models\varphi^{\text{Ack}(f)}$.
Let $\II' = (A',\assign')$ be a model of $\Phi$ such that $\II'(\varphi^{\text{Ack}(f)}) = \bot$.
Then, in particular, $\II'(\varphi^\text{FC}) = \top$ and $\II'(\varphi^\text{flat}) = \bot$.
Now, consider the interpretation $\II = (A,\assign')$, where $A$ is obtained from $A'$ by replacing the function $f^{A'}$ by the function
$f^A$ that agrees with $f^{A'}$ except for $f^A( \II'(\tau_1^{(i)}), \dots,  \II'(\tau_n^{(i)})) = \II'(c_i)$, for all $i$, where $\tau_k^{(i)} = \acker(t_k^{(i)})$.
Note that this function is well-defined since $\II'$ is a model of $\varphi^\text{FC}$.
Since the symbol $f$ does not appear in $\Phi$, the interpretations of all terms in $\Phi$ remain unchanged, and thus, $\II(\Phi) = \II'(\Phi) = \top$.
Furthermore, $\II(\varphi) = \II'(\varphi^\text{flat}) = \bot$, and consequently $\Phi \not\models\varphi$.
\end{proof}

Since the axioms $\AA$ of preadditive semicategories only contain arithmetic function symbols,
Proposition~\ref{prop ackermann cons} implies that Ackermann's reduction can be used to remove all non-arithmetic function symbols from quantifier-free operator statements.
Combining this with Herbrand's theorem and Herbrandisation to eliminate quantifiers, we now have a systematic way of transforming any operator statement $\varphi$ into an arithmetic operator statement $\varphi'$ such that $\varphi$ is universally true if and only if $\varphi'$ is.
In Section~\ref{sec:all-operator-statements}, we describe this reduction formally, and show how it gives rise to a semi-decision procedure (Procedure~\ref{algo semantic consequence}) for verifying universal truth of arbitrary operator statements based on algebraic computations. 

Procedure~\ref{algo semantic consequence} is based on the fact that we can characterise universal truth of arithmetic operator statements 
algebraically via ideal membership of noncommutative polynomials.
In the following section, we establish this characterisation.


\section{An algebraic approach to universal truth of arithmetic operator statements}
\label{sec:arithmetic-statements}

In this section, we characterise universal truth of arithmetic operator statements by polynomial ideal membership.
To this end, we introduce the notion of \emph{idealisation} (Section~\ref{sec:idealisation}) and use this concept to state and prove our main result (Theorem~\ref{thm nullstellensatz}).
Before doing this, we first formally describe the translation of terms and literals in arithmetic operator statements into noncommutative polynomials.

For the remainder of this section, we fix a signature $\Sigma = (\bm{O} \times \bm{O},\bm{C},\bm{F},\sigma)$.

\subsection{Translation into polynomials}
\label{sec:terms-to-polynomials}

Let $\varphi \in \Formulas(\Sigma)$ be an arithmetic operator statement.
We associate to every variable and nonzero constant symbol appearing in $\varphi$ an indeterminate, that is, we consider the set 
\[
	X_\varphi = \left\{ x \in \Var \cup (\bm{C} \setminus \Zero) \mid x \text{ appears in }\varphi \right\}
\]
of indeterminates.
Note that we do not assign indeterminates to the distinguished zero constants $0_{u,v}$.
With this, we define a translation function $T_\varphi$, that maps every term in $\varphi$ to a noncommutative polynomial in $\Z\<X_\varphi>$, as follows:
\begin{alignat*}{3}
	&T_\varphi(0_{u,v}) &&=\quad &&0 \\
	&T_\varphi(x) &&= &&x \quad\text{ if } x \in \Var \cup (\bm{C} \setminus \Zero) \\
	&T_\varphi(-t) &&= &&- T_\varphi(t) \\
	&T_\varphi(s + t) &&= &&T_\varphi(s) + T_\varphi(t) \\
	&T_\varphi(s \cdot t) &&= &&T_\varphi(s) \cdot T_\varphi(t)
\end{alignat*}


In order to simplify the notation, we identify a term $t$ appearing in $\varphi$ with its image $T_\varphi(t)$ under the translation map.
Furthermore, if $s \approx t$ or $s \not\approx t$ is a literal in $\varphi$, we write $s - t$ for the polynomial $T_\varphi(s) - T_\varphi(t)$.

\begin{example}\label{ex:idealisation}
Consider a signature $\Sigma = (\bm{O} \times \bm{O},\bm{C},\bm{F},\sigma)$ with $\bm{O} = \{u,v\}$ and sort function $\sigma$ such that
\[
	\sigma(a) = (u,v), \qquad \sigma(b) = (v,u), \qquad \sigma(c) = (u,v),  \qquad \sigma(d) = (v,v),
\]
where $a,b,c,d \in \Var$.

Let $\varphi \in \Formulas(\Sigma)$ be the following arithmetic operator statement:
\[
	\varphi = \left(
	a\cdot b\cdot a \approx 0_{u,v} + c \wedge
	d\cdot a \not\approx c
	\right)
	\rightarrow
	\left(
	a \cdot b - d \not\approx 0_{v,v}
	 \wedge 
	a \cdot b\cdot c \approx c \cdot b\cdot a
	\right).
\]

Using the indeterminates $X_{\varphi} = \{a,b,c,d\}$, the four literals appearing in $\varphi$ are translated into the following polynomials in $\Z\<a,b,c,d>$:
\begin{align*}
	aba - c, && da - c, && ab -d, && abc - cba.
\end{align*}
\end{example}

Note that, thus far, positive and negative literals are handled equivalently.
Both $s \approx t$ and $s \not\approx t$ are translated into the same polynomial $s - t$.
A distinction between positive and negative literals will become apparent in the next section.

\subsection{Idealisation}
\label{sec:idealisation}

Idealisation is a process that allows to characterise universal truth of arithmetic operator statements via ideal membership in a free algebra $\Z\<X>$.
In what follows, we consider $\wedge$ and $\vee$ as associative and commutative operations, that is, 
$\varphi \wedge \psi = \psi \wedge \varphi$ and $\varphi \wedge (\psi \wedge \rho) = (\varphi \wedge \psi) \wedge \rho$, and analogously for $\vee$.
We recall that every quantifier-free formula $\varphi$ can be transformed into a logically equivalent formula of the form
\begin{equation}\label{eq syntax}
	\bigwedge_{i=1}^{m} \left( \bigvee_{j=1}^{n_i} s_{i,j} \not\approx t_{i,j} \vee \bigvee_{k=1}^{n'_i} p_{i,k} \approx q_{i,k} \right),
\end{equation}
with terms $s_{i,j}, t_{i,j}, p_{i,k}, q_{i,k}$.
In the above formula, either of the two disjunctions can also be empty, that is, it is possible that either $n_i = 0$ or $n_i' = 0$, but not both.

We recall that a formula of the form~\eqref{eq syntax} is said to be in \emph{conjunctive normal form (CNF)}.
It is a conjunction of \emph{clauses}, where a clause is a disjunction of literals.
A formula can have several CNFs.

One way to obtain a CNF of a quantifier-free formula $\varphi$ is to apply to~$\varphi$ exhaustively each of the following sets of rewrite rules in the given order, see also, e.g.,~\cite[Sec.~7.5.2]{russell2010artificial}:
\begin{enumerate}
	\item Eliminate implications:
	\begin{equation}\label{eq:rewrite-1}
		\psi_1 \rightarrow \psi_2 \;\rightsquigarrow\; \lnot \psi_1 \vee \psi_2
	\end{equation}
	\item Move $\lnot$ inwards (i.e., compute a negation normal form):
	\begin{equation}\label{eq:rewrite-2}
		\lnot \lnot \psi  \;\rightsquigarrow\; \psi,  \qquad
		\lnot (\psi_1 \wedge \psi_2) \;\rightsquigarrow\; \lnot \psi_1 \vee \lnot \psi_2, \qquad
	\lnot (\psi_1 \vee \psi_2) \;\rightsquigarrow\; \lnot \psi_1 \wedge \lnot \psi_2
	\end{equation}
	\item Distribute $\vee$ over $\wedge$:
	\begin{equation}\label{eq:rewrite-3}
		\psi \vee (\psi_1 \wedge \psi_2) \;\rightsquigarrow\; (\psi \vee \psi_1) \wedge (\psi \vee \psi_2)
	\end{equation}
\end{enumerate}

We note that the above rules apply modulo associativity and commutativity of $\wedge, \vee$.
This process yields a unique result, which is a CNF of $\varphi$.
We denote it by $\CNF(\varphi)$.
Also note that this transformation preserves the semantics of $\varphi$.

\begin{lemma}\label{lemma cnf}
Any quantifier-free formula $\varphi$ is logically equivalent to $\CNF(\varphi)$.
\end{lemma}

Based on the conjunctive normal form, we define the idealisation.
We first discuss the special case of arithmetic clauses.

\begin{definition}
Let $C = \bigvee_{j=1}^{n} s_{j} \not\approx t_{j} \vee \bigvee_{k=1}^{n'} p_{k} \approx q_{k} \in \Formulas(\Sigma)$ be an arithmetic clause.
The \emph{idealisation} $\Id(C)$ of $C$ is the following predicate considered as a statement in the free algebra $\Z\<X_C>$:
\begin{align}\label{eq idealisation}
	\Id(C) \;\coloneqq\; \text{there exists } 1 \leq k \leq n' \text{ such that }p_{k} - q_{k} \in (s_{1} - t_{1},\dots, s_n - t_n ).
\end{align}
\end{definition}

\begin{remark}
Formally, the idealisation $\Id(C)$ has to be considered as a nullary predicate, parametrised by the arithmetic clause $C$.   
Thus, for each arithmetic clause $C$, we obtain a different predicate.
This is to avoid having to deal with higher-order logics in order to define a predicate of formulas.
\end{remark}

To motivate this definition, write $C$ in the equivalent form $\bigwedge_j s_{j} \approx t_{j} \rightarrow \bigvee_{k} p_{k} \approx q_{k}$.
This shows that $C$ is universally true if and only if at least one identity $p_{k} \approx q_{k}$ can be derived from all $s_{j} \approx t_{j}$ and from the linearity axioms encoded in $\Axioms$. 
This fact is described by $\Id(C)$, neglecting sorts.

To extend this definition to arbitrary arithmetic operator statements $\varphi$, we consider the normal form $\CNF(\varphi)$ and use the fact that a conjunction 
of clauses is satisfied if and only if each clause is satisfied individually.
Intuitively, each clause in the CNF corresponds to one case in a proof that $\varphi$ is universally true.
Based on this, the idealisation of $\varphi$ is given by conjunctively combining the idealisations of all clauses $C$ in $\CNF(\varphi)$.

\begin{definition}
Let $\varphi \in \Formulas(\Sigma)$ be an arithmetic operator statement. 
The \emph{idealisation} $\Id(\varphi)$ of $\varphi$ is the predicate
\[
	\Id(\varphi) \coloneqq \bigwedge_{\substack{C \text{ clause}\\[0.1em] \text{of } \CNF(\varphi)}} \Id(C).
\]
\end{definition}

\begin{example}
We compute the truth value of the idealisation $\Id(\varphi)$ of the arithmetic operator statement $\varphi$ from Example~\ref{ex:idealisation}, that is,
\[
	\varphi = \left(
	a\cdot b\cdot a \approx 0_{u,v} + c \wedge
	d\cdot a \not\approx c
	\right)
	\rightarrow
	\left(
	a \cdot b - d \not\approx 0_{v,v}
	 \wedge 
	a \cdot b\cdot c \approx c \cdot b\cdot a
	\right).
\]
First, note that, $\CNF(\varphi) = C_{1} \wedge C_{2}$, with clauses 
\begin{align*}
	C_{1} &= a\cdot b\cdot a \not\approx 0_{u,v} + c \;\vee\; d\cdot a \approx c \;\vee\; a \cdot b - d \not\approx 0_{v,v}, \\
	C_{2} &= a\cdot b\cdot a \not\approx 0_{u,v} + c \;\vee\; d\cdot a \approx c \;\vee\; a \cdot b\cdot c \approx c \cdot b\cdot a.
\end{align*}
The truth values of the idealisations of these two clauses correspond to those of the polynomial ideal membership statements
\begin{align*}
	\Id(C_{1}) &\,\equiv\, da - c \in \left(a b a - c, a b - d\right), \\
	\Id(C_{2}) &\,\equiv\, da - c \in (a b a - c) \;\text{ or }\; a bc - c b a \in (aba- c),
\end{align*}
which both hold since
\begin{alignat*}{2}
	da - c &= (aba - c) - (ab - d)a &&\in \left(a b a - c, a b - d\right), \\
	abc - cba &= (aba - c)ba - ab(aba - c) &&\in \left( aba - c\right).
\end{alignat*}
Therefore, we have $\Id(C_{1}) \equiv \Id(C_{2}) \equiv \top$, and thus, by definition, also
\begin{align*}
	\Id(\varphi) \equiv \Id(C_{1}) \wedge \Id(C_{2}) \equiv \top \wedge \top \equiv \top.
\end{align*}

\end{example}

In the succeeding section, we will prove that the idealisation of an arithmetic operator statement $\varphi$ is true if and only if $\varphi$ is universally true.
To this end, we state the following lemma, which relates the semantics of the basic logical connectives $\wedge$ and $\vee$ to the notion of idealisation.

\begin{lemma}\label{lemma properties idealisation}
For arithmetic operator statements $\varphi, \psi \in \Formulas(\Sigma)$ the following hold:
\begin{enumerate}
	\item $\Id(\varphi) \equiv \top$ implies $\Id(\varphi \vee \psi) \equiv \top$;
	\item $\Id(\varphi) \equiv \top$ and $\Id(\psi) \equiv \top$ implies $\Id(\varphi \wedge \psi) \equiv \top$;
\end{enumerate}
\end{lemma}

\begin{proof}
	\begin{enumerate}
		\item Let $\CNF(\varphi) = \bigwedge_i C_i$ and $\CNF(\psi) = \bigwedge_j D_j$ with clauses $C_i, D_j$.
		Each clause of $\CNF(\varphi \vee \psi)$ is of the form $C_i \vee D_j$.
		By assumption $\Id(C_i) \equiv \top$ for all $i$. 
		Then the statement follows from the fact that the ideal corresponding to $\Id(C_i)$ is contained in the one corresponding to $\Id(C_i  \vee D_j)$
		and the set of candidate members for $\Id(C_i)$ is a subset of those for $\Id(C_i  \vee D_j)$.
		
		\item By assumption, the idealisation of each clause of $\CNF(\varphi)$ and $\CNF(\psi)$ is true.
		 Then the statement follows from the fact that each clause of $\CNF(\varphi \wedge \psi)$ is a clause from $\CNF(\varphi)$ or from $\CNF(\psi)$.\qedhere
	\end{enumerate}
\end{proof}

\subsection{Characterising universal truth via ideal membership}
\label{sec:main-theorem}

In this section, we relate the semantic notion of universal truth of an arithmetic operator statement $\varphi$ to the ideal theoretic statement $\Id(\varphi) \equiv \top$.
The main result is the following theorem.
Recall that we have fixed a signature $\Sigma = (\bm{O} \times \bm{O},\bm{C},\bm{F},\sigma)$.

\begin{theorem}\label{thm nullstellensatz}
An arithmetic operator statement $\varphi \in \Formulas(\Sigma)$ is universally true if and only if $\Id(\varphi) \equiv \top$.
\end{theorem}

\begin{proof}
The ``if''-part follows from Lemma~\ref{lemma syntax to semantics} below and the ``only if''-part from Lemma~\ref{lemma semantics to syntax}.
\end{proof}

Theorem~\ref{thm nullstellensatz} reduces universal truth to the verification of finitely many ideal memberships in $\Z\<X>$.
Since ideal membership in $\Z\<X>$ is semi-decidable (Proposition~\ref{prop ideal membership}) this immediately yields a semi-decision procedure for verifying universal truth of arithmetic operator statements.
In Section~\ref{sec:all-operator-statements}, we generalise this procedure to all operator statements. 

\begin{remark}
We make some remarks about Theorem~\ref{thm nullstellensatz}.
\begin{enumerate}
	\item The idealisation $\Id(\varphi)$ is independent of the signature $\Sigma$.
	This is because the corresponding polynomial ideal membership statements, and thus also the truth value of $\Id(\varphi)$, are the same for all signatures $\Sigma$ in which $\varphi$ can be formulated.
	In particular, they are independent of the sorts appearing in $\varphi$.
	Consequently, Theorem~\ref{thm nullstellensatz} tells us that a single computation with polynomials shows the
	universal truth of $\varphi$ in \emph{all} signatures in which $\varphi$ can be formulated.
	\item The idealisation $\Id(\varphi)$ is defined with respect to a specific CNF of $\varphi$.
	However, Theorem~\ref{thm nullstellensatz} implies that the truth value of $\Id(\varphi)$ is independent of the concrete CNF used.
	Thus, any CNF of $\varphi$, in fact, any CNF that is $\Axioms$-equivalid to $\varphi$, can be used to compute the truth value of $\Id(\varphi)$.
	This can provide computational advantages.
\end{enumerate}
\end{remark}

Theorem~\ref{thm nullstellensatz} is a generalisation of~\cite[Thm.~32]{RRH21}.
We state a specialisation of this result below, as it will play a crucial role in proving the sufficiency part of Theorem~\ref{thm nullstellensatz}.
The language of~\cite{RRH21} differs from the language of this work in many points.
Consequently, the following result, which we formulate in the language of this work, looks very different to the original in~\cite{RRH21}.
To facilitate the comparison between the two versions, we state the main differences in Remark~\ref{remark:comparison}.
We also note that Theorem~\ref{cor rrh21} is a formal version of Theorem~\ref{cor:quasi-identities}.

\begin{theorem}[{\cite[Thm.~32]{RRH21}}]\label{cor rrh21}
Let $C \in \Formulas(\Sigma)$ be an arithmetic operator statement of the form
\[
	C = \left(s_1 \approx t_1 \wedge \dots \wedge s_n \approx t_n\right) \rightarrow p \approx q.
\]
If $\Id(C) \equiv \top$, then $C$ is universally true.
\end{theorem}

\begin{remark}
\label{remark:comparison}
\mbox{}
\begin{enumerate}
	\item We note that~\cite[Thm.~32]{RRH21} is phrased for $R$-linear categories, where $R$ is a commutative ring with unit.  
	By specialising $R = \Z$ the result follows for preadditive categories. 
	By examining the proof it becomes clear that it also holds for preadditive semicategories,
	as the proof does not rely on any properties of identity morphisms, which is the only difference between $\Z$-linear categories and preadditive semicategories.

	 \item In this work, sources and targets of morphisms are encoded in form of sorts via the sort function $\sigma$.
	 In contrast to this,~\cite{RRH21} uses a directed multigraph, also known as \emph{quiver}, for this.
	 A quiver is a natural choice for representing morphisms, however, it is unsuited for representing functions that use those morphisms as
	 arguments.
	 In the language of this work, quivers are good for representing the sorts of variables and constant symbols but not for representing the sorts of function symbols.
	 We note that~\cite{RRH21} does not consider function symbols other than the arithmetic function symbols, whose sorts are predetermined anyway.

	 \item An important notion in~\cite{RRH21} are \emph{compatible polynomials}.
	 Being compatible means that the polynomials respect the restrictions imposed by the sources and targets of the morphisms they model.
	 In the terminology of this work, one can say that compatible polynomials are precisely those elements $s - t$ that originate from translating literals $s \approx t$ or $s \not\approx t$.
	 
	\item Another difference between~\cite{RRH21} and this work is how a semantic meaning is assigned to syntactic objects.
	Here we use the concept of interpretations for this, while \cite{RRH21} uses \emph{representations} of a quiver.
	It turns out that the latter can be considered as a special case of the former.

\end{enumerate}
\end{remark}

The remainder of this section is dedicated to proving Theorem~\ref{thm nullstellensatz}.
More precisely, Lemma~\ref{lemma syntax to semantics} proves the sufficiency part of Theorem~\ref{thm nullstellensatz}
by reducing the case of arbitrary formulas to the case of clauses and using Theorem~\ref{cor rrh21}.
For the other implication, we use the fact that the sequent calculus LK$^=$ is sound and complete, i.e., that $\AA \models \varphi$ if and only if $\AA \vdash \varphi$.
In particular, we consider the idealisation of certain sequents and show in Lemma~\ref{lemma preserve contradiction} that all relevant sequent rules of LK$^{=}$ (see~\ref{sec:formal-comp}) preserve the property of these idealisations being true.
In order to handle the presence of the axioms $\AA$ in sequents, we prove Lemma~\ref{lemma remove axioms} which captures one of the main advantages of our approach -- the axioms $\AA$ can be neglected when performing idealisations.
Finally, we combine Lemma~\ref{lemma preserve contradiction} and~\ref{lemma remove axioms} with Herbrand's theorem to prove the necessity part of Theorem~\ref{thm nullstellensatz} in Lemma~\ref{lemma semantics to syntax}.

\begin{lemma}\label{lemma syntax to semantics}
Let $\varphi \in \Formulas(\Sigma)$ be an arithmetic operator statement.
If $\Id(\varphi) \equiv \top$, then~$\varphi$ is universally true.
\end{lemma}

\begin{proof}
Recall that $\varphi$ being universally true means that $\Axioms \models \varphi$.

First, we consider the case where $\varphi$ is a clause, that is,
\[
	\varphi = \bigvee_{j=1}^n s_j \not\approx t_j \vee \bigvee_{k=1}^{n'} p_k \approx q_k.
\]
By definition of the idealisation, $\Id(\varphi) \equiv \top$ if and only if there exists $1 \leq k \leq n'$ such that $p_{k} - q_{k} \in (s_{1} - t_{1},\dots, s_n - t_n)$.
Reordering the literals in $\varphi$, we can,
without loss of generality, assume that $k = 1$.
Let $C = \bigvee_{j=1}^n s_j \not\approx t_j \vee p_{1} \approx q_{1}$ and note that $\Id(C) \equiv \top$.
Since $C$ is logically equivalent to
\[
	\left(s_1 \approx t_1 \wedge \dots \wedge s_{n} \approx t_{n}\right) \rightarrow p_{1} \approx q_{1},
\]
Theorem~\ref{cor rrh21} implies that $C$ is universally true, i.e., that $\Axioms \models C$, but then also $\Axioms \models C \vee \bigvee_{k=2}^{n'} p_k \approx q_k$,
and the result follows.

For arbitrary $\varphi$, let $\CNF(\varphi) = \bigwedge_{i=1}^m C_i$ with clauses $C_i$.
Then $\Id(\varphi) \equiv \top$ implies $\Id(C_i) \equiv \top$ for all $i = 1,\dots,m$.
By the previous discussion, $\Axioms \models C_i$ for all $i = 1,\dots,m$.
Now, the semantics of $\wedge$ implies $\Axioms \models \bigwedge_{i=1}^m C_i$, showing that $\CNF(\varphi)$ is universally true.
Then the result follows from Lemma~\ref{lemma cnf}.
\end{proof}

Let $\Gamma, \Delta \subseteq \Formulas(\Sigma)$ be finite multisets of arithmetic operator statements.
In the following, we define an idealisation of sequents $\Gamma \vdash \Delta$.
More precisely, we denote 
\[
	\Id\left(\Gamma \vdash \Delta\right) \;\coloneqq\; \Id\left(\bigvee_{\gamma \in \Gamma} \lnot \gamma \vee \bigvee_{\delta \in \Delta} \delta\right).
\]
We say that an axiom sequent rule\def\fCenter{\ \vdash\ }\AxiomC{$$}\UnaryInf$\Gamma \fCenter \Delta$\DisplayProof
\emph{preserves validity} if $\Id(\Gamma \vdash \Delta) \equiv \top$.
Similarly, we say that a sequent rule of the form \def\fCenter{\ \vdash\ }\Axiom$\Gamma' \fCenter \Delta'$\UnaryInf$\Gamma \fCenter \Delta$\DisplayProof 
\emph{preserves validity} if $\Id(\Gamma' \vdash \Delta') \equiv \top$ implies $\Id(\Gamma \vdash \Delta) \equiv \top$.
Finally, a sequent rule of the form \def\fCenter{\ \vdash\ }\Axiom$\Gamma_1 \fCenter \Delta_1$\Axiom$\Gamma_2 \fCenter \Delta_2$\BinaryInf$\Gamma \fCenter \Delta$\DisplayProof
\emph{preserves validity} if \mbox{$\Id(\Gamma_1 \vdash \Delta_1) \equiv \top$ and $\Id(\Gamma_2 \vdash \Delta_2) \equiv \top$}
implies $\Id(\Gamma \vdash \Delta) \equiv \top$.

\begin{lemma}\label{lemma preserve contradiction}
All sequent rules of the sequent calculus LK$^{=}$ (see~\ref{sec:formal-comp}) not involving quantifiers preserve validity when applied to arithmetic operator statements.
\end{lemma}

\begin{proof}

All sequent rules in LK$^{=}$ contain multisets $\Gamma$ and $\Delta$ which are comprised of all formulas that are irrelevant for the particular rule.
To take care of these formulas, we denote in the following $\Phi = \bigvee_{\gamma \in \Gamma} \lnot\gamma \vee \bigvee_{\delta \in \Delta} \delta$.

\textbf{Axioms.}

For (Ax), we have to show that
$\Id(\Gamma, t \approx t' \vdash \Delta, t \approx t') \equiv \top$, which, by definition, means showing that $\Id(\Phi \vee t \not\approx t' \vee t \approx t') \equiv \top$.
The idealisation of $t \approx t' \vee t \not\approx t'$ corresponds to the polynomial assertion $t-t' \in (t-t')$, which is clearly true.
Thus, $\Id(t \approx t' \vee t \not\approx t') \equiv \top$, and the result follows from the first part of Lemma~\ref{lemma properties idealisation}.

For (Ref), we have to show that $\Id(\Gamma \vdash \Delta, t \approx t) \equiv \top$, which, by definition, means showing that $\Id(\Phi \vee t \approx t) \equiv \top$.
Note that $\Id(t \approx t) \equiv \top$, because it corresponds to the ideal membership $t - t = 0 \in (\emptyset) = \{0\}$.
With this, the result follows from the first part of Lemma~\ref{lemma properties idealisation}.

\textbf{Structural rules.}

Preserving validity for the rules ($\text{W}\ \vdash$) and ($\vdash\ \text{W}$) follows from the first part of Lemma~\ref{lemma properties idealisation}.
The other two rules ($\text{C}\ \vdash$) and ($\vdash\ \text{C}$) only remove duplicates of formulas.
This leaves the ideals and the sets of candidate members generated during the idealisation unchanged, and thus, also these rules preserve validity.

\textbf{Propositional rules.} 

For the $\lnot$-rules, preserving validity translates into the trivial statements
``$\Id(\Phi \vee \varphi) \equiv \top$ implies $\Id(\Phi \vee \varphi) \equiv \top$'' and ``$\Id(\Phi \vee \lnot \varphi) \equiv \top$ implies $\Id(\Phi \vee \lnot \varphi) \equiv \top$'' respectively.

For ($\vee\ \vdash$), we denote $\Psi = \Phi \vee \lnot (\varphi \vee \psi)$, $\Phi_1 = \Phi \vee \lnot \varphi$, and $\Phi_2 = \Phi \vee \lnot \psi$.
We have to show that $\Id(\Psi) \equiv \top$ assuming $\Id(\Phi_1) \equiv \top$ and $\Id(\Phi_2) \equiv \top$.
Note that applying the rewrite rules~\eqref{eq:rewrite-1} and~\eqref{eq:rewrite-2} for the CNF computation to $\Psi$ yields 
\begin{equation}\label{eq:proof-1}
	\Phi' \vee (\varphi' \wedge \psi'),
\end{equation}
where $\Phi'$ is a negation normal form of $\Phi$ and $\varphi'$, $\psi'$ are negation normal forms of $\lnot \varphi$ and $\lnot \psi$ respectively.
Analogously, applying~\eqref{eq:rewrite-1} and~\eqref{eq:rewrite-2} to $\Phi_1 \wedge \Phi_2$ yields 
\begin{equation}\label{eq:proof-2}
	(\Phi' \vee \varphi') \wedge (\Phi' \vee \psi').
\end{equation}
Applying now the distributivity rule~\eqref{eq:rewrite-3} to~\eqref{eq:proof-1} gives~\eqref{eq:proof-2}, showing that $\CNF(\Psi) = \CNF(\Phi_1 \wedge \Phi_2)$.
Thus also $\Id(\Psi) = \Id(\Phi_1 \wedge \Phi_2)$, and the result follows from the second part of Lemma~\ref{lemma properties idealisation}. 

For ($\vdash\ \vee$), we obtain the trivial statement ``$\Id(\Phi \vee \varphi \vee \psi) \equiv \top$ implies $\Id(\Phi \vee \varphi \vee \psi) \equiv \top$''.

The proofs for $\wedge$ are dual to the proofs for $\vee$.

Using the fact that $\CNF(\varphi \rightarrow \psi) = \CNF(\lnot \varphi \vee \psi)$ the proofs for $\rightarrow$ are analogous to the proofs for $\vee$.

\textbf{Equational rule.} 

Denote $\Psi =  \Phi \vee \lnot \varphi[x \mapsto y] \vee t \not\approx t'$,
where $y$ is a new variable that does not appear in any of the formulas.
We have to show that $\Id(\Psi[y \mapsto t']) \equiv \top$ implies $\Id(\Psi[y \mapsto t]) \equiv \top$.

To this end, note that, if $\CNF(\Psi) = \bigwedge_{i} C_{i}$ with clauses $C_{i}$, then 
 $\CNF(\Psi[y \mapsto t]) = \bigwedge_{i} C_{i}[y \mapsto t]$, and analogously for $y \mapsto t'$.
This, together with the definition of the idealisation, implies that it suffices to show that $\Id(C[y \mapsto t']) \equiv \top$ implies $\Id(C[y \mapsto t]) \equiv \top$ for
every clause $C$ of $\CNF(\Psi)$.
Thus, for the following, we fix a clause $C$ of $\CNF(\Psi)$.

By definition of the CNF-transformation, $C$ contains the literal $t \not\approx t'$.
More precisely, we can write
\begin{align*}
	C = t \not\approx t' \vee \bigvee_{j=1}^{n} s_{j} \not\approx t_{j} \vee \bigvee_{k=1}^{n'} p_{k} \approx q_{k},
\end{align*}
with terms $s_{j}, t_{j}, p_{k}, q_{k}$. 
When computing the idealisation of $C$, every literal containing the variable $y$ gets translated into a polynomial depending on $y$. 
We emphasise this dependency by writing $a(y) - b(y)$ for the polynomial obtained from translating a literal $a \approx b$ or $a \not\approx b$ 
that may involve $y$, even if it actually does not. 
Then, by definition, the idealisation of $C$ is true if and only if there exists $1 \leq k \leq n'$ such that
\begin{align*}
	p_{k}(y) - q_{k}(y) \in (t - t', s_{1}(y) - t_{1}(y),\dots, s_n(y) - t_n(y)).
\end{align*}
With this, considering the idealisation of $C[y \mapsto t]$ only requires to replace the variable $y$ in the polynomials above by the translation of the term $t$.   
That is, the idealisation of $C[y \mapsto t]$ is true if and only if there exists $1 \leq k \leq n'$ such that
\begin{align*}
	p_{k}(t) - q_{k}(t) \in J \coloneqq (t - t', s_{1}(t) - t_{1}(t),\dots, s_n(t) - t_n(t)).
\end{align*}
Analogously, $\Id(C[y \mapsto t']) \equiv \top$ if and only if there exists $1 \leq k \leq n'$ such that
\begin{align*}
	p_{k}(t') - q_{k}(t') \in J' \coloneqq (t - t', s_{1}(t') - t_{1}(t'),\dots, s_n(t') - t_n(t')).
\end{align*}

Now, since $t-t'$ is a generator of both ideals $J$ and $J'$, we have that, for any polynomial $f$, $f(t)$ lies in any of these ideals if and only if $f(t')$ does.
Therefore, in fact, $J = J'$ and $p_{k}(t) - q_{k}(t) \in J$ if and only if $p_{k}(t') - q_{k}(t') \in J'$ for any $1 \leq k \leq n'$.
This implies that $\Id(C[y \mapsto t]) \equiv \top$ if and only if $\Id(C[y \mapsto t']) \equiv \top$.
\end{proof}

The following lemma captures a main advantage of our approach.
It asserts that instantiations of the axioms $\AA$ can be neglected when performing the idealisation.
Recall, from Section~\ref{sec:herbrand}, that $H(\AA)$ denotes the Herbrand expansion of $\AA$.

\begin{lemma}\label{lemma remove axioms}
Let $\varphi \in \Formulas(\Sigma)$ be an arithmetic operator statement. 
Then $\Id(\lnot \alpha \vee \varphi) \equiv \Id(\varphi)$ for all ground instances $\alpha \in H(\AA)$.
\end{lemma}

\begin{proof}
Note that $\lnot \alpha = s \not \approx t$ for ground terms $s,t \in \Ground(\Sigma)$ and that 
\[
	\CNF(\lnot \alpha \vee \varphi) = \CNF(s \not\approx t \vee \varphi) = \bigwedge_i (s \not \approx t \vee C_i),
\]
where $\CNF(\varphi) = \bigwedge_i C_i$.
The ideal generated in the idealisation of each clause $s \not \approx t \vee C_i$ is 
the ideal obtained from $C_{i}$ plus the additional generator $s-t$.
The statement now follows from the fact that $s - t = 0$ since the formulas in $\AA$ only describe basic linearity properties that hold in any ring.
\end{proof}

\begin{lemma}\label{lemma semantics to syntax}
Let $\varphi \in \Formulas(\Sigma)$ be an arithmetic operator statement. 
If $\varphi$ is universally true, then $\Id(\varphi) \equiv \top$.
\end{lemma}

\begin{proof}
Recall that $\varphi$ being universally true means that $\Axioms \models \varphi$.
Since $\varphi$ is quantifier-free, all variables that appear in $\varphi$ are free.
In fact, by considering the universal closure of $\varphi$
and using Herbrandisation to replace all variables by new constant symbols, 
we can, without loss of generality, assume that $\varphi$ is a ground sentence.
Note that this preserves universal truth.

The axioms $\Axioms$ form a set of sentences in Skolem normal form.
Consequently, Theorem~\ref{thm:herbrand} yields the existence of $\alpha_1,\dots,\alpha_k \in H(\Axioms)$ such that
\begin{align}\label{eq proof idealisation}
	\alpha_1 \wedge \dots \wedge \alpha_k \vdash \varphi.
\end{align}
Here we used the fact that $\varphi$ does not contain any variables, and therefore $H(\varphi) = \{\varphi\}$.
Since all formulas in~\eqref{eq proof idealisation} are quantifier-free, there exists a formal proof in the sequent calculus LK$^{=}$ that does not use any of the quantification rules.
As, by Lemma~\ref{lemma preserve contradiction}, all other rules preserve validity, we have that 
\[
	\top \equiv \Id(\alpha_1 \wedge \dots \wedge \alpha_k \vdash \varphi) = \Id(\lnot \alpha_1 \vee \dots \vee \lnot\alpha_k \vee \varphi).
\]
Then repeated application of Lemma~\ref{lemma remove axioms} yields the desired result.
\end{proof}

\section{A semi-decision procedure for all operator statements}
\label{sec:all-operator-statements}

Using Herbrand's theorem and Ackermann's reduction (Section~\ref{sec:herbrand-ackermann}), we can systematically reduce the case of verifying that an arbitrary operator statement $\varphi$ is universally true to the previously discussed case of arithmetic operator statements.
The following steps give an overview on how this can be done.
They can be considered as an adaptation of Gilmore's algorithm~\cite{gilmore1960proof}.

\begin{enumerate}
	\item Bring $\varphi$ into Herbrand normal form, denoted by $\varphi^H$.
	\item Let $\varphi_1,\varphi_2,\dots$ be an enumeration of the Herbrand expansion $H(\varphi^H)$.
	\item Let $n = 1$.
	\item Form the formula $\psi_n = \bigvee_{i=1}^n \varphi_i$.\label{step loop}
	\item Remove all non-arithmetic function symbols from $\psi_n$ using Ackermann's reduction.
	Denote the obtained formula by $\psi_n^\text{Ack}$ (note that $\psi_n^\text{Ack}$ is an arithmetic operator statement).
	\item If $\Id(\psi_n^\text{Ack}) \equiv \top$, then $\varphi$ is universally true.
	Otherwise, increase $n$ by $1$ and go to step~\ref{step loop}.\label{step check}
\end{enumerate}

Since our language is expressive enough to encode the word problem for semigroups, which is known to be undecidable~\cite[Thm.~4.5]{Dav58}, we cannot expect to obtain a correct and complete algorithm that terminates on any input.
The best we can hope for is a semi-decision procedure that terminates if and only if an operator statement is indeed universally true.
However, the steps above, as phrased now, still have a subtle flaw that stops them from even being a semi\nobreakdash-decision procedure. 

The conditional check in step~\ref{step check} requires to decide certain ideal memberships.
While \emph{verifying} ideal membership of noncommutative polynomials is always possible in finite time, \emph{disproving} it is generally not.
Consequently, verifying that the condition in step~\ref{step check} is false is generally not possible in finite time.
In cases where this is required, the procedure above cannot terminate -- even if $\varphi$ is indeed universally true.

To overcome this flaw and to obtain a true semi-decision procedure, we have to interleave the computations done for different values of $n$.
Procedure~\ref{algo semantic consequence} shows one way how this can be done.
It essentially follows the steps described above, except that it only performs finitely many steps to check if $\Id(\psi_n^\text{Ack}) \equiv \top$ for each $n$.

\floatname{algorithm}{Procedure}
\begin{algorithm}[h]
\renewcommand{\algorithmicrequire}{\textbf{Input:}}
\renewcommand{\algorithmicensure}{\textbf{Output:}}
\caption{Semi-decision procedure for verifying universal truth}\label{algo semantic consequence}
\begin{algorithmic}[1]
\Require{a signature $\Sigma$ and an operator statement $\varphi \in \Formulas(\Sigma)$}
\Ensure{$\top$ if and only if $\varphi$ is universally true; otherwise infinite loop}

\State{$\varphi^H \leftarrow$ Herbrand normal form of $\varphi$}
\State{$\varphi_1,\varphi_2,\dots \leftarrow$ an enumeration of $H(\varphi^H)$}\label{line enumeration}
\State{$\psi_1 \leftarrow \varphi_1$}
\State{$\psi_1^\text{Ack} \leftarrow$ Ackermann's reduction of $\psi_1$ removing all non-arith.\ function symbols}
\For{$n \leftarrow 1,2\dots$}
	\For{$k \leftarrow 1,\dots,n$}
	\If{$\Id(\psi_{k}^\text{Ack}) \equiv \top$ can be verified with $n$ steps of an ideal membership semi\nobreakdash-decision procedure}\label{line ideal checks}
		\textbf{return } $\top$
	\EndIf
	\EndFor
	\If{$\varphi^H$ is a ground sentence}\Comment{in this case $H(\varphi^H) = \{\varphi^H\} = \{ \varphi_1  \}$}
		\State{$\psi_{n+1}^\text{Ack} \leftarrow \psi_{n}^\text{Ack}$}
	\Else
		\State{$\psi_{n+1} \leftarrow \psi_{n} \vee \varphi_{n+1}$}\label{line instance}
		\State{$\psi_{n+1}^\text{Ack} \leftarrow$ Ackermann's reduction of $\psi_{n+1}$ removing all non-arith.\ function symbols}
	\EndIf
    \EndFor
\end{algorithmic}
\end{algorithm}

\begin{remark}
We make some remarks about Procedure~\ref{algo semantic consequence}.
\begin{enumerate}
	\item The procedure treats the case where the Herbrand normal form of the input formula is a ground sentence separately.
	In this case, the computation can be simplified since the Herbrand expansion is a singleton.
	Then $\psi_n^\text{Ack} = \varphi_1^\text{Ack} = (\varphi^H)^\text{Ack}$ for all $n$ since $\varphi^H$ is the only element in $H(\varphi^H)$.
	
	\item Line~\ref{line ideal checks} contains the term \emph{step of a procedure}. 
	Thereby we mean any (high- or low-level) set of instructions of the procedure that can be executed in \emph{finite time}.
	We note that the vagueness of the previous sentence is deliberate, because it is just a means to split a (possibly) infinite
	computation into parts that are guaranteed to be finite.
	The particular details on how this is done play no role in the correctness of Procedure~\ref{algo semantic consequence}.
	
	\item If $\varphi$ is not universally true, then the procedure runs into an infinite loop.
	However, it can be modified to also terminate in some cases when $\varphi$ is not universally true, see the discussion 
	immediately before Section~\ref{sec:AE-quasi-identities}.
\end{enumerate}
\end{remark}

\begin{theorem}\label{thm algo}
Procedure~\ref{algo semantic consequence} terminates and outputs $\top$ if and only if $\varphi$ is universally true.
\end{theorem}

\begin{proof}
Recall that $\varphi$ being universally true means that $\Axioms \models \varphi$.
Also, note that $\Axioms$ is a set of sentences in Skolem normal form.
Thus, by Proposition~\ref{prop herbrandisation} and Theorem~\ref{thm:herbrand}, $\Axioms \models \varphi$ if and only if $\Axioms \models \psi_N$ for some large enough $N \in \N$.
Also, note that $\psi_N$ is a quantifier-free operator statement.
Hence, Proposition~\ref{prop ackermann cons} yields that $\Axioms \models \varphi$ is equivalent to $\Axioms \models \psi_N^\text{Ack}$.
Finally, since $\psi_N^\text{Ack}$ is an arithmetic operator statement, Theorem~\ref{thm nullstellensatz} implies 
that $\Axioms \models \varphi$ is equivalent to $\Id(\psi_N^\text{Ack}) \equiv \top$.
To summarise, $\Axioms \models \varphi$ is equivalent to the condition in line~\ref{line ideal checks} being satisfied at some point.
This shows that, if the algorithm terminates, then $\Axioms \models \varphi$.
The other implication follows from Proposition~\ref{prop ideal membership} and the fact that the computations for each $n$ can be executed in a finite amount of time.
\end{proof}

In the following, we give a partial answer to the question ``When can we prove that $\varphi$ is \emph{not} universally true?''
or in other words ``When does $\varphi$ \emph{not} follow from the axioms of preadditive semicategories?''.
We assume, without loss of generality, that $\varphi$ is in Herbrand normal form. 
First, we note that if $H(\varphi)$ is infinite, then it is not possible to deduce that $\varphi$ is not universally true using our approach
as this would require to argue over all (infinitely many) finite subsets of $H(\varphi)$.
If $H(\varphi)$ is not infinite, then $H(\varphi) = \{\varphi\}$. 
In this case, proving that $\varphi$ is not universally true boils down to verifying that $\Id(\varphi^\text{Ack}) \equiv \bot$, which is the case if and only if certain ideal memberships do not hold.
Although generally impossible, there are cases in which ideal membership can be disproven (for example, when the ideal is homogeneous or admits a finite Gr\"obner basis). 
In these situations, based on Theorem~\ref{thm nullstellensatz}, we can prove that $\varphi$ is not universally true.

\subsection{Arithmetic $\forall\exists$-quasi-identities}
\label{sec:AE-quasi-identities}

A class of operator statements that we commonly encounter in practice, and for which Procedure~\ref{algo semantic consequence} can be simplified drastically, 
is the following class of \emph{arithmetic $\forall\exists$-quasi-identities}.
For the following definition and theorem, we fix a signature $\Sigma = (\bm{O} \times \bm{O},\bm{C},\bm{F},\sigma)$.

\begin{definition}
An \emph{arithmetic $\forall\exists$-quasi-identity} is an operator statement of the form
\begin{align*}
	\forall \bm{x} \exists \bm{y} : \bigwedge_{i=1}^{m} s_{i}(\bm{x}) \approx t_{i}(\bm{x}) \rightarrow p(\bm{x},\bm{y}) \approx q(\bm{x},\bm{y}) \in \Sent(\Sigma),
\end{align*}
with terms $s_{i}(\bm{x}), t_{i}(\bm{x}), p(\bm{x},\bm{y}), q(\bm{x},\bm{y})$ that only contain arithmetic function symbols.
\end{definition}

\begin{remark}
In the definition of arithmetic $\forall\exists$-quasi-identities, we also allow degenerate cases without universally or existentially quantified variables, that is, where $\bm{x} = \emptyset$ or $\bm{y} = \emptyset$.
However, note that an arithmetic $\forall\exists$-quasi-identity has to be a sentence, meaning that all variables that appear have to be quantified by a quantifier.
\end{remark}

The following theorem shows that, for arithmetic $\forall\exists$-quasi-identities, Procedure~\ref{algo semantic consequence} essentially reduces to a single polynomial ideal membership.

\begin{theorem}\label{thm:AE-identities}
An arithmetic $\forall\exists$-quasi-identity
\begin{align*}
	\forall \bm{x} \exists \bm{y} : \bigwedge_{i=1}^{m} s_{i}(\bm{x}) \approx t_{i}(\bm{x}) \rightarrow p(\bm{x},\bm{y}) \approx q(\bm{x},\bm{y}),
\end{align*}
where $\bm{y} = y_{1},\dots,y_{k}$, is universally true if and only if there exist terms $\bm{z} = z_{1}(\bm{x}),\dots,z_{k}(\bm{x})$ depending only on $\bm{x}$ such that $\sigma(z_{j}) = \sigma(y_{j})$, for all $j =1,\dots,k$, and such that the ideal membership
\begin{align*}
	p(\bm{x},\bm{z}) - q(\bm{x},\bm{z}) \in \big( s_{1}(\bm{x}) - t_{1}(\bm{x}), \dots, s_{m}(\bm{x}) - t_{m}(\bm{x}) \big)
\end{align*}
holds in the free algebra $\Z\<\bm{x}>$.
\end{theorem}

\begin{proof}
We show that applying Procedure~\ref{algo semantic consequence} to an arithmetic $\forall\exists$-quasi-identity corresponds precisely to 
checking the condition in the theorem.
First, note that the Herbrand normal form of an arithmetic $\forall\exists$-quasi-identity, is the formula itself, 
only with the universal quantifier removed -- here we assume, without loss of generality,
that the new constant symbols are given the same names as the variables $\bm{x}$ they replace.
Furthermore, the Herbrand expansion of this normal form is the set of all formulas of the form 
\[
	\bigwedge_{i=1}^{m} s_{i}(\bm{x}) \approx t_{i}(\bm{x}) \rightarrow p(\bm{x},\bm{z}) \approx q(\bm{x},\bm{z})
\]
with ground terms $\bm{z} = z_{1}(\bm{x}),\dots,z_{n}(\bm{x})$ that may involve the constants $\bm{x}$ and satisfy $\sigma(z_{j}) = \sigma(y_{j})$ for all $j =1,\dots,k$.
Note that the symbols $\bm{x}$ now have to be considered as constant symbols, as the variables have been replaced during Herbrandisation, and thus, the formulas in the Herbrand expansion are indeed ground.
Then, Procedure~\ref{algo semantic consequence} has to verify whether the idealisation $\Id(\varphi_{1} \vee \dots \vee \varphi_{n})$ is true for some $n \in \N$ and elements $\varphi_{l}$ in the Herbrand expansion, $l = 1,\dots,n$.
Note that Ackermann's reduction is not required because arithmetic $\forall\exists$-quasi-identities only contain arithmetic function symbols by definition.
If we write each $\varphi_{l} = \bigwedge_{i=1}^{m} s_{i}(\bm{x}) \approx t_{i}(\bm{x}) \rightarrow p(\bm{x},\bm{z_{l}}) \approx q(\bm{x},\bm{z_{l}})$, then a conjunctive normal form of $\varphi_{1} \vee \dots \vee \varphi_{n}$ is given by the single clause
\[
	\bigvee_{i=1}^{m} s_{i}(\bm{x}) \not\approx t_{i}(\bm{x}) \vee \bigvee_{l=1}^{n} p(\bm{x},\bm{z_{l}}) \approx q(\bm{x},\bm{z_{l}}).
\]
Finally, the idealisation of this clause is true if and only if the ideal membership 
\[
	p(\bm{x},\bm{z_{l}}) - q(\bm{x},\bm{z_{l}}) \in \left(s_{1}(\bm{x}) - t_{1}(\bm{x}), \dots, s_{m}(\bm{x}) - t_{m}(\bm{x})\right)
\]
holds for some $1 \leq l \leq n$. \qedhere
\end{proof}

\subsection{Computational aspects}\label{sec computational aspects}

The performance of Procedure~\ref{algo semantic consequence} strongly depends on the order in which the elements in the Herbrand expansion are enumerated.
In the following, we present a few basic techniques to find appropriate instances needed for a verification that an operator statement $\varphi$ is universally true that turned out to be useful in practice.
We also discuss other aspects relevant for applying the procedure in practice. 
Some of these techniques deviate quite strongly from how Procedure~\ref{algo semantic consequence} is phrased, but they can have a drastic impact on the efficiency of the procedure.

In practical applications, the operator statement $\varphi$ under consideration is typically an implication of the form $\varphi = \bigwedge_{i} \alpha_i \rightarrow \gamma$,
where $\alpha_i$ represent different assumptions and $\gamma$ forms a claimed property. 
In the following, we assume that $\varphi$ is of this form.

Universal identities are sentences of the form $\forall \bm{x} : s(\bm{x}) \approx  t(\bm{x})$ where $s,t$ are terms depending on the variables $\bm{x} = x_1,\dots,x_n$.

Such formulas often describe very basic properties (such as linearity/multiplicativity of a function symbol, commutativity of certain operators, etc.). 
If an $\alpha_{i}$ is a universal identity, then, formally, each application of such an assumption in a proof requires a new instantiation of the formula $\varphi$.
This increases the size of the resulting formula, making further computations more involved.
To avoid this problem, we can remove the assumption $\forall \bm{x} : s(\bm{x}) \approx  t(\bm{x})$ from $\varphi$ and
treat it implicitly by considering a universal rewrite rule $s \mapsto t$ that maps every instantiation $s(t_1,\dots,t_n)$ to $t(t_1,\dots,t_n)$.
This rewrite rule is then applied exhaustively to all formulas that occur during the computation.
We note that this approach need not be complete, but it is definitely sound, and in practice we were, thus far, still always able to prove the desired statement using this method.

Next, we discuss situations where the claim $\gamma$ is an existential identity, that is, where $\gamma$ is a sentence of the form $\exists \bm{y} : p(\bm{y}) \approx q(\bm{y})$.
Such formulas can describe, for instance, the existence of a solution to an equation and finding suitable instantiations for them is usually a hard task.
In our setting, however, appropriate instantiations can often be found automatically with the help of algebraic techniques.
For example, in the setting of Theorem~\ref{thm:AE-identities}, finding a suitable instantiation corresponds to finding a polynomial 
\begin{align*}
	p(\bm{x},\bm{z}) - q(\bm{x},\bm{z}) \in \big( s_{1}(\bm{x}) - t_{1}(\bm{x}), \dots, s_{m}(\bm{x}) - t_{m}(\bm{x}) \big),
\end{align*}
with unknown terms $\bm{z} = z_{1}(\bm{x}), \dots, z_{k}(\bm{x})$. 
If something about the structure of $p(\bm{x},\bm{z}) - q(\bm{x},\bm{z})$ is known, for example, a factor in a factorisation or if the sought-after element is known to be homogeneous,
then dedicated algebraic techniques, mostly based on Gr\"obner basis methods, can be used to systematically search for such polynomials in an ideal.
We refer to~\cite{HRR22} for an explanation of these methods, along with illustrative examples of successfully automated proofs of operator statements involving existential claims.

Once suitable instantiations have been found, they are combined disjunctively into one formula and Ackermann's reduction is performed to remove non-arithmetic function symbols.
It is worth noting that, in automated reasoning, the congruence closure algorithm~\cite[Sec.~4.1]{KS16} is typically favoured over Ackermann's reduction for handling equality with function symbols in the ground case.
Unfortunately, applying congruence closure straightforwardly presents challenges in our context.
The problem arises from a crucial step in the congruence closure algorithm that requires deciding equality of ground terms.
In our application, we also have to take into consideration the axioms $\AA$ of preadditive semicategories when performing congruence closure, 
which implies the need to decide equality of ground terms \emph{modulo the axioms $\AA$}.
It is not clear to us how the congruence closure algorithm could be adapted to do this efficiently, allowing to exploit our algebraic setting.

Therefore, we have decided to work with Ackermann's reduction, as this technique applies to our setting without any modifications.
In practice, however, we made the observation that combining Ackermann's reduction with a very rudimentary version of congruence closure can be beneficial.
Thus, in practice, when an operator statement $\varphi$ contains a unary non-arithmetic function symbol $f$, we, as a first step, extend assumptions of the form
\[
	\bigwedge_{i} s_i \approx t_i \rightarrow \bigvee_{j} p_j \approx q_j
\]
to
\[
	\bigwedge_{i} s_i \approx t_i \rightarrow \bigvee_{j}\ (p_j \approx q_j \wedge f(p_j) \approx f(q_j)),
\]
where $f$ is applied to all terms of the correct sort.
In particular, plain identities $p \approx q$ can be extended to $p \approx q \wedge f(p) \approx f(q)$. 
This process is clearly sound and often helps to automatically generate appropriate function instantiations, speeding up the computation in this way.

When forming the idealisation of the resulting arithmetic operator statement $\varphi$, the formula has to be transformed into CNF.
This transformation can lead to an exponential increase in size and to exponentially many ideal membership verifications during the idealisation.
Classically in automated theorem proving, this exponential blow-up is avoided by a so-called Tseitin transformation, see, for example,~\cite[Sec.~1.3]{KS16}, which transforms a formula into a new formula in CNF by introducing new variables, causing only a linear increase in size.
However, while this process preserves satisfiability, it need not preserve validity, the notion we are interested in. 
Instead of trying to adapt the Tseitin transformation to our setting, we instead explain an approach based on an incremental computation of the idealisation. 
It follows from the following observation.
If $\varphi = (\alpha \vee \beta) \rightarrow \gamma$ is an arithmetic operator statement, then 
\[
	\Id(\varphi) = \Id(\lnot \alpha \vee \gamma) \wedge \Id(\lnot \beta \vee \gamma).
\]
In such situations, it can make sense to compute the idealisation incrementally. 
First, we form $\Id(\gamma)$ and if $\Id(\gamma) \equiv \top$, then also $\Id(\varphi) \equiv \top$ (by Lemma~\ref{lemma properties idealisation}) and we are done.
Only if $\Id(\gamma) \equiv \top$ cannot be verified fast enough (i.e., within some fixed time frame), we compute $\Id(\lnot \alpha \vee \gamma)$ and $\Id(\lnot \beta \vee \gamma)$.
The main advantage of this idea is that it can be applied recursively in case of several assumptions (see Figure~\ref{fig implications} for an illustration).
This is because $\CNF(\bigwedge_i \alpha_i \rightarrow \gamma)$ and $\CNF(\alpha_1 \rightarrow (\alpha_2 \rightarrow \dots \rightarrow (\alpha_m \rightarrow \gamma) \dots ))$ 
are the same, and thus, lead to the same idealisation.
While, in the worst case, this technique still requires exponentially many ideal membership tests, it in practice usually allows to drastically reduce the number of polynomial computations. 

\def\vsep{-2}
\def\hsep{2}
\begin{figure}
\centering
\begin{tikzpicture}
	\node at (0,0)			 {$\Id(\gamma) \equiv\ ?$};
	
	\node at (0-1.8*\hsep,\vsep) 	{$\Id(\lnot \alpha_1 \vee \gamma) \equiv \top$};
	\node at (0+1.8*\hsep,\vsep) 	{$\Id(\lnot\beta_1 \vee \gamma) \equiv\ ?$};
	
	\node at (-1.8*\hsep - 0.8*\hsep,2*\vsep)  {\color{gray}$\begin{aligned}\Id(\lnot\alpha_1 &\vee \lnot\alpha_2 \vee \gamma) \\&\equiv \top\end{aligned}$};
	\node at (-1.8*\hsep + 0.8*\hsep,2*\vsep)  {\color{gray}$\begin{aligned}\Id(\lnot\alpha_1 &\vee \lnot\beta_2 \vee \gamma) \\&\equiv \top\end{aligned}$};
		 
	\node at (1.8*\hsep - 0.8*\hsep,2*\vsep)  {$\begin{aligned}\Id(\lnot\beta_1 &\vee \lnot\alpha_2 \vee \gamma) \\&\equiv \top\end{aligned}$};
	\node at (1.8*\hsep + 0.8*\hsep,2*\vsep) {$\begin{aligned}\Id(\lnot\beta_1 &\vee \lnot\beta_2 \vee \gamma) \\&\equiv \top\end{aligned}$};
	
	\node[color=gray] at (-2.6*\hsep - 0.4*\hsep,2.6*\vsep)  {\dots};
	\node[color=gray] at (-2.6*\hsep + 0.4*\hsep,2.6*\vsep) {\dots};
	\node[color=gray] at (-1*\hsep - 0.4*\hsep,2.6*\vsep)  {\dots};
	\node[color=gray] at (-1*\hsep + 0.4*\hsep,2.6*\vsep) {\dots};
	\node[color=gray] at (1*\hsep - 0.4*\hsep,2.6*\vsep)  {\dots};
	\node[color=gray] at (1*\hsep + 0.4*\hsep,2.6*\vsep) {\dots};
	\node[color=gray] at (2.6*\hsep - 0.4*\hsep,2.6*\vsep)  {\dots};
	\node[color=gray] at (2.6*\hsep + 0.4*\hsep,2.6*\vsep) {\dots};
			
	\draw[->] (0-0.3*\hsep,0+0.3*\vsep) to (0-1.2*\hsep,0+0.7*\vsep);
	\draw[->] (0+0.3*\hsep,0+0.3*\vsep) to (0+1.2*\hsep,0+0.7*\vsep);
	
	\draw[color=gray,->] (-1.8*\hsep - 0.2*\hsep,0+1.3*\vsep) to (-1.8*\hsep - 0.8*\hsep ,0+1.7*\vsep);
	\draw[color=gray,->] (-1.8*\hsep + 0.2*\hsep,0+1.3*\vsep) to (-1.8*\hsep + 0.8*\hsep, 0+1.7*\vsep);
	
	\draw[->] (1.8*\hsep - 0.2*\hsep, 0+1.3*\vsep) to (1.8*\hsep - 0.8*\hsep, 0+1.7*\vsep);
	\draw[->] (1.8*\hsep + 0.2*\hsep, 0+1.3*\vsep) to (1.8*\hsep + 0.8*\hsep, 0+1.7*\vsep);
	
	\draw[color=gray,->] (-1.8*\hsep - 0.2*\hsep,0+1.3*\vsep) to (-1.8*\hsep - 0.8*\hsep ,0+1.7*\vsep);
	\draw[color=gray,->] (-1.8*\hsep + 0.2*\hsep,0+1.3*\vsep) to (-1.8*\hsep + 0.8*\hsep, 0+1.7*\vsep);
	
	\draw[color=gray,->] (-2.6*\hsep - 0.1*\hsep,2.3*\vsep) to (-2.6*\hsep - 0.3*\hsep ,2.5*\vsep);
	\draw[color=gray,->] (-2.6*\hsep + 0.1*\hsep,2.3*\vsep) to (-2.6*\hsep + 0.3*\hsep, 2.5*\vsep);
	\draw[color=gray,->] (-1*\hsep - 0.1*\hsep,2.3*\vsep) to (-1*\hsep - 0.3*\hsep ,2.5*\vsep);
	\draw[color=gray,->] (-1*\hsep + 0.1*\hsep,2.3*\vsep) to (-1*\hsep + 0.3*\hsep, 2.5*\vsep);
	\draw[color=gray,->] (1*\hsep - 0.1*\hsep,2.3*\vsep) to (1*\hsep - 0.3*\hsep ,2.5*\vsep);
	\draw[color=gray,->] (1*\hsep + 0.1*\hsep,2.3*\vsep) to (1*\hsep + 0.3*\hsep, 2.5*\vsep);
	\draw[color=gray,->] (2.6*\hsep - 0.1*\hsep,2.3*\vsep) to (2.6*\hsep - 0.3*\hsep ,2.5*\vsep);
	\draw[color=gray,->] (2.6*\hsep + 0.1*\hsep,2.3*\vsep) to (2.6*\hsep + 0.3*\hsep, 2.5*\vsep);
	
\end{tikzpicture}
\caption{Illustration on how to incrementally compute the idealisation of $\left((\alpha_1 \vee \beta_1) \wedge (\alpha_2 \vee \beta_2) \wedge (\alpha_3 \vee \beta_3)\right) \rightarrow \gamma$.
Grey colored computations are avoided.
Here, the question mark ``?'' indicates that the idealisation could not be verified to be true within some fixed time frame.}
\label{fig implications}
\end{figure}
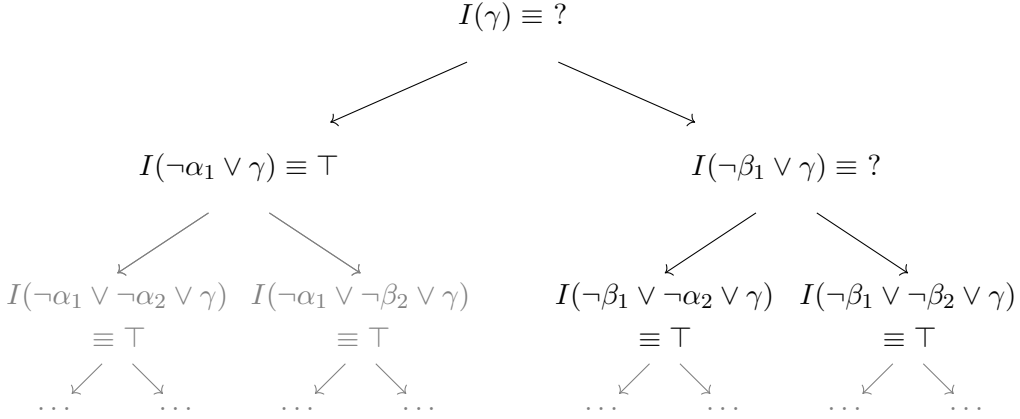

No matter whether the optimisation discussed above is used or not, the crucial step of Procedure~\ref{algo semantic consequence} requires the verification of 
finitely many ideal memberships.
We note that such a polynomial computation is independent of the involved terms' sorts, providing a clear advantage over other verification approaches (such as sequent calculi)
that require additional checks to ensure that the sort restrictions are respected at every step.
Additionally, ideal membership $f \in (f_1,\dots,f_r)$ can be certified easily by providing a representation of $f$ as a linear combination of the generators $f_1,\dots,f_r$.
Such a certificate can be computed automatically and serves as a proof that can be checked easily and independently of how it was obtained.
We refer to~\cite{HRR19,Hof20} on how such certificates can be computed.


\section{Applications}\label{sec:applications}

In the different applications described in this section, we mostly deviate from Procedure~\ref{algo semantic consequence} 
and use the approaches discussed in Section~\ref{sec computational aspects} in order to speed up the computations.
In particular, we make extensive use of algebraic techniques for finding appropriate instantiations of existentially quantified variables. 
Also, many statements that arise in the applications can be phrased as arithmetic $\forall\exists$-quasi-identities and can thus be treated using Theorem~\ref{thm:AE-identities}. 

\subsection{Introductory example -- existence of the Moore-Penrose inverse}

As an introductory example to illustrate the workings of the framework,
we consider a simple statement about the existence of the Moore-Penrose inverse for morphisms in involutive preadditive semicategories.
In this setting, a morphism need not have a Moore-Penrose inverse.
The following part of~\cite[Lem.~3]{PR81} provides a sufficient condition for its existence.

\begin{lemma}\label{lemma MP-inverse}
Let $A\colon U \to V$ be a morphism in an involutive preadditive semicategory.
If there exist morphisms $P,Q \colon U \to V$ such that $PA^{*}A = A = AA^{*}Q$, then $A$ has a Moore-Penrose inverse.
\end{lemma}

In the following, we describe how to prove Lemma~\ref{lemma MP-inverse} using the framework introduced in this work and Procedure~\ref{algo semantic consequence}.

\paragraph{Step 1: Choosing a signature}

As a first step, we have to choose a suitable signature.
To this end, we fix distinct $u,v \in \ObS$ and consider $\Sigma = (\bm{O} \times \bm{O},\bm{C},\bm{F},\sigma)$ with
\begin{enumerate}
	\item $\bm{O} = \{u,v\}$;
	\item $\bm{C} = \{0_{u,u}, 0_{u,v}, 0_{v,u}, 0_{v,v} \}$;
	\item $\bm{F} = \{{}^*_{u,u}, {}^*_{u,v}, {}^*_{v,u}, {}^*_{v,v}\}  \cup \{-_{\alpha,\beta}, +_{\alpha,\beta}, \cdot_{\alpha,\beta,\gamma} \mid \alpha, \beta, \gamma \in \bm{O}\} $;
	\item $\sigma$ satisfies the conditions mentioned in the beginning of Section~\ref{sec:operator-statements-and-universal-truth} (assigning the expected sorts to zero constants and arithmetic function symbols);
	\item $\sigma$ satisfies $\sigma({}^*_{\alpha,\beta}) = (\alpha,\beta) \to (\beta,\alpha)$ for $\alpha,\beta \in \bm{O}$
\end{enumerate}

Besides the required zero constants and arithmetic function symbols, the signature contains function symbols ${}^*_{u,u}, {}^*_{u,v}, {}^*_{v,u}, {}^*_{v,v}$ representing the involution. 
To simplify the notation, we will denote these function symbols all by the same symbol ${}^*$ in the following.
We also require $\sigma$ to assign the expected sorts to these function symbols.  

\paragraph{Step 2: Translating the statement}

After fixing a signature, we can translate the statement about morphisms into a first-order formula.
To keep the notation uncluttered, we omit the sorts of the variables in the formula.
We note that $\sigma(a) = \sigma(p) = \sigma(q) = (u,v)$ and $\sigma(x) = (v,u)$.

In the following, we write $st$ for $s \cdot t$ to keep the expressions shorter.
Then Lemma~\ref{lemma MP-inverse} can be translated into the formula
\[
	\varphi = \forall a,p,q \,\exists x : \varphi_\text{Asm} \rightarrow \varphi_\text{MP},
\]
where
\begin{align*}
	\varphi_\text{Asm} &= pa^{*}a \approx a  \;\wedge\; aa^{*}q \approx a  \;\wedge\; (pa^{*}a)^{*} \approx a^{*}  \;\wedge\; (aa^{*}q)^{*} \approx a^{*}\\
	\varphi_\text{MP} &= a x a \approx a \;\wedge\; xax \approx x \;\wedge\; (ax)^* \approx ax \;\wedge\; (xa)^* \approx xa.
\end{align*}

The formula consists of the assumptions $\varphi_\text{Asm}$, encoding the identities $PA^{*}A = A = AA^{*}Q$,
and of the claimed property $\varphi_\text{MP}$, expressing the existence of a Moore-Penrose inverse of $A$.

Formally, the assumptions should also include the properties~\eqref{eq:involution} of the involution function symbols.
Since these properties are very basic universal identities, we, as described in Section~\ref{sec computational aspects}, omit them and instead 
consider the universal rewrite rules $(t^*)^* \mapsto t$ and $(st)^* \mapsto t^* s^*$ that we use to simplify all occurring terms.
Additionally, we have added the adjoint identities $(pa^{*}a)^{*} \approx a^{*} \wedge (aa^{*}q)^{*} \approx a^{*}$ to our assumptions encoding $(PA^{*}A)^{*} = A^{*} = (AA^{*}Q)^{*}$.

\paragraph{Step 3: Applying Procedure~\ref{algo semantic consequence}}

Since $\varphi$ is already a sentence in prenex normal form, we can immediately start with the Herbrandisation to remove the universal quantifier.
In this process, we replace the variables $a,p,q$ by new constant symbols, which we denote for simplicity also by $a,p,q$.
This yields the Herbrand normal form
\[
	\varphi^H = \exists x : \varphi_\text{Asm} \rightarrow \varphi_\text{MP}.
\]

Then we consider an enumeration $\varphi_1,\varphi_2,\dots$ of the Herbrand expansion $H(\varphi^H)$ where
\begin{align}\label{eq instantiation}
	\varphi_1 &= (\varphi_\text{Asm} \rightarrow \varphi_\text{MP})[x \mapsto q^{*} a p^{*}]
\end{align}
and set $\psi_1 = \varphi_1$.

\begin{remark}
Without any prior knowledge, one would most likely choose a very different enumeration of $H(\varphi^H)$,
for example, instantiating the variable $x$
first only with constants and then successively use more complex terms. 
In order to avoid unnecessary computations, we use an enumeration that starts with the appropriate instantiation.
 
We note that, in this example, the appropriate instantiation of $x$ can be found automatically using algebraic techniques.
For example, 
one can compute a Gr\"obner basis of the ideal generated by all identities appearing in $\varphi_\text{Asm} \rightarrow \varphi_\text{MP}$ w.r.t.\ an elimination ordering for $x$.
This Gr\"obner basis then contains an element of the form $x - q^{*}ap^{*}$ indicating that $x \mapsto q^* a p^*$ is the appropriate instantiation.
We refer to~\cite[Sec.~3]{casc-2023}, where this approach and the corresponding computations are described in more detail.
\end{remark}

Now, Ackermann’s reduction is used to remove the function symbol ${}^*$ from $\psi_1$.
After applying the rewrite rules $(st)^* \mapsto t^* s^*$, $(t^*)^* \mapsto t$ exhaustively, the 
only instances of this function symbol are $a^*, p^*, q^*$.
We replace these terms by new constants, which we denote by the same symbols.
This yields a new formula 
\[
	\psi_1^\text{flat} = \psi_{\text{Asm}} \rightarrow \psi_{\text{MP}},
\]
where
\begin{align*}
	\psi_\text{Asm} &= pa^{*}a \approx a  \;\wedge\; aa^{*}q \approx a  \;\wedge\; a^{*}ap^{*} \approx a^{*}  \;\wedge\; q^{*}aa^{*} \approx a^{*}\\
	\psi_\text{MP} &= a q^{*}ap^{*} a \approx a \;\wedge\; q^{*}ap^{*}aq^{*}ap^{*} \approx q^{*}ap^{*} \\
				&\wedge\; pa^{*}qa^{*} \approx aq^{*}ap^{*} \;\wedge\; a^*pa^{*}q \approx q^{*}ap^{*}a.
\end{align*}

Furthermore, we add the functional consistency constraints
\begin{align*}
	\psi_1^\text{FC} =\; (a \approx p \rightarrow a^* \approx p^*) \wedge (a \approx q \rightarrow a^* \approx q^*) \wedge (p \approx q \rightarrow p^* \approx q^*).
\end{align*}
The resulting formula is $\psi_1^\text{Ack} = \psi_1^\text{FC} \rightarrow \psi_1^\text{flat}$.

Directly computing the idealisation of $\psi_1^\text{Ack}$ would require 32 ideal membership tests.
To reduce the number of polynomial computations, we apply the optimisation discussed in Section~\ref{sec computational aspects} and compute $\Id(\psi_1^\text{Ack})$ incrementally. 
Every implication $\alpha \rightarrow \beta$ corresponds to a disjunction $\lnot \alpha \vee \beta$.
So, we remove all implications from the assumptions, which means removing $\psi_1^\text{FC}$ and considering only $\psi_1^\text{flat}$.
We form the CNF of the simplified formula, yielding the 4 clauses:
\begin{align*}
	C_1 &=  \psi_{\text{Asm}} \vee a q^{*}ap^{*} a \approx a, \\
	C_2 &=  \psi_{\text{Asm}} \vee q^{*}ap^{*}aq^{*}ap^{*} \approx q^{*}ap^{*}, \\
	C_3 &=  \psi_{\text{Asm}} \vee  pa^{*}qa^{*} \approx aq^{*}ap^{*}, \\
	C_4 &=  \psi_{\text{Asm}} \vee  a^*pa^{*}q \approx q^{*}ap^{*}a.
\end{align*}
As can be seen, we obtain one clause corresponding to each of the 4 identities in $\psi_\text{MP}$.

It turns out that the idealisations of all four clauses $C_1,\dots,C_4$ are true.
For example, the representation 
\begin{align*}
	a q^{*}ap^{*} a - a &= (pa^{*}a - a)(1 - p^{*}a) + (1-aq^{*})(aa^{*}q - a)p^{*}a \\
					&+ p(a^{*}ap^{*} - a^{*})a + a(q^{*}aa^{*} - a^{*})qp^{*}a			
\end{align*}
shows that
\begin{align*}
	a q^{*}ap^{*} a - a \in \left(pa^{*}a - a, aa^{*}q - a, a^{*}ap^{*} - a^{*}, q^{*}aa^{*} - a^{*} \right),
\end{align*}
and thus certifies $\Id(C_{1}) \equiv \top$.

Similarly, it can be certified that also the other 3 idealisations are true.
Therefore, by definition, $\Id(\psi_1^\text{flat}) \equiv \top$, and thus also $\Id(\psi_{1}^{\text{Ack}}) \equiv \top$.
This finishes the proof of Lemma~\ref{lemma MP-inverse}.
Notably, this approach required only 4 ideal membership tests compared to 32 if we had translated $\psi_1^\text{Ack}$ directly.
We note that all required polynomial computations are rather simple and lead to ideal membership certificates of similar size and complexity as the one shown above.

\subsection{Moore-Penrose case study}

Two of the authors have conducted a case study to illustrate the effectivity of the framework in combination with our computer algebra package \texttt{operator\_gb} on statements regarding the Moore-Penrose inverse~\cite{casc-2023}.
In the following, we largely reproduce the discussion in~\cite[Sec.~6]{casc-2023}.

In this case study, the proofs of a variety of theorems on the Moore\nobreakdash-Penrose inverse have been automated, ranging from classical facts in the Handbook of Linear Algebra~\cite[Sec.~I.5.7]{hogben2013handbook} over important characterisations of the Moore-Penrose inverse $(AB)^{\dag}$ of a product of linear operators $A,B$ on Hilbert spaces~\cite{DD10} to very recent improvements~\cite{cvetkovic2021algebraic} of a classical result by Hartwig~\cite{Har86} that were found with the help of our software.
A Jupyter notebook containing the automated proofs of all statements is available at
\begin{center}
\url{https://cocalc.com/georeg/Moore-Penrose-case-study/notebook}.
\end{center}

The case study covers the first 25 facts in the section on the Moore-Penrose inverse in the Handbook of Linear Algebra~\cite[Sec.~I.5.7]{hogben2013handbook}.
Among these 25 statements, five cannot be treated within the framework as they contain properties that cannot be expressed in terms of algebraic identities of operators.
These are, for example, properties of matrix entries or analytic statements involving limits.
Additionally, three statements can only be partially handled for the same reason.
The remaining 17 statements, along with those parts of the three statements mentioned before that can be treated within the framework, can all be translated into statements about polynomial ideal memberships.
Using our software, these statements are proven fully automatically.
The corresponding polynomial computations take place in ideals generated by up to 70 polynomials in up to 18 indeterminates.
The proof of each statement takes less than one second, and the computed representations, certifying the required ideal memberships, consist of up to 226 terms.

The case study also includes Theorems~2.2 -- 2.4 in~\cite{DD10}, which characterise situations when the reverse order law $(AB)^\dag = B^\dag A^\dag$ of bounded linear operators $A,B$ on Hilbert spaces holds.
Using our software, these theorems can be automatically proven in less than five seconds altogether, generating algebraic proofs consisting of up to 279 terms.
It is noteworthy that, unlike the original proofs in~\cite{DD10}, which rely on matrix forms of bounded linear operators that are induced by some decompositions of Hilbert spaces, these algebraic proofs do not require any structure on the underlying spaces besides a certain cancellability assumption and the property that inclusions of ranges $\mathcal{R}(A) \subseteq \mathcal{R}(B)$ of operators $A,B$ can be translated into the existence of a factorisation $A = BX$, for some operator $X$. 
In Hilbert and Banach spaces, this is the well-known factorisation property in Douglas' lemma~\cite{douglas}.
Thus, our algebraic proofs generalise the results from bounded operators on Hilbert spaces to morphisms in arbitrary preadditive semicategories meeting these assumptions.

Finally, the case study contains fully automated proofs of Theorem~2.3 and~2.4 in our joint work~\cite{cvetkovic2021algebraic}, which provide necessary and sufficient conditions for the triple reverse order law $(ABC)^\dag = C^\dag B^\dag A^\dag$ to hold, where $A,B,C$ are elements in a ring with involution.
These theorems can be considered as significant improvements of a classical result by Hartwig~\cite{Har86}. 
We discuss them in more detail in the succeeding section.

\subsection{Reverse order law}
\label{sec:ROL}

For a classical inverse in a group, it is easy to show that $(ab)^{-1} = b^{-1} a^{-1}$ for all group elements $a,b$.
In the setting of generalised inverses, this \emph{reverse order law} need no longer hold unconditionally.
The \emph{reverse order law problem}, originally posed by Greville~\cite{Gre66}, asks for necessary and sufficient conditions for the reverse order law to hold. 

Since Greville, who considered this problem in the case of the Moore-Penrose inverse of the product of two matrices,
the reverse order law has been an area of extensive research branching in several directions (products of more than two matrices, different classes of generalised inverses, different settings such as operator algebras, $C^*$-algebras, rings, etc.).
One important result is due to Hartwig~\cite{Har86}, who studied the reverse order law for the Moore-Penrose inverse of a product of three matrices.

In joint work with colleagues~\cite{cvetkovic2021algebraic}, we have generalised this result by Hartwig using partly automated polynomial computations based on the framework developed in~\cite{RRH21}.
Using the framework of this work, these new results can now be proven fully automatically with the help of our software package \texttt{operator\_gb}.
The automated proofs of these statements can be found in the Jupyter notebook mentioned in the previous section.

To give an impression of the size of the computation,
we note that the assumptions in the proof of~\cite[Thm.~2.3]{cvetkovic2021algebraic} consist of up to 24 identities in 22 variables 
and the certificates of the ideal memberships that are computed consist of up to 80 terms.
Using our software package, the polynomial computations to find the appropriate instantiations and to obtain the certificates take in total about 15 seconds on a regular laptop.

\subsection{Diagram chases}
\label{sec:diagram-chases}

Abelian categories form a special class of preadditive categories where kernels and cokernels of morphisms exist and have certain nice properties.
They are very important structures providing a natural setting for homological algebra and have many applications in pure category theory and algebraic geometry, see, for example,~\cite[Ch.~VIII]{Mac13} for a textbook exposition and further details.
In particular, many statements in homological algebra, for example, can be phrased as \emph{diagram lemmas} in abelian categories.
Commonly, these diagram lemmas are proven using a technique called \emph{diagram chasing} and different approaches have been developed to automate such computations~\cite{Him20,Pos22}.
The framework introduced in this paper encompasses the language of abelian categories and allows to automate proofs of diagram lemmas by computations with noncommutative polynomials.
We note that this work is, thus far, still very much in its beginnings. 

For example, a very central property that regularly appears in diagram lemmas is that of a sequence $U \xrightarrow{f} V \xrightarrow{g} W$ of morphisms $f\colon U \to V$ and $g \colon V \to W$ being \emph{exact}.
This property can be expressed in terms of the language of our framework by exploiting, for example, the characterisation below.
For a proof of this fact and for characterisations of further classical properties, we refer to~\cite[Prop.~7.3.2]{Hof23}.

\begin{proposition}
A sequence $U \xrightarrow{f} V \xrightarrow{g} W$ is exact if and only if $g \circ f = 0$ and, for all $x \colon T \to V$ with $g \circ x = 0$, there exist $y \colon S \to U$ and $e \colon S \to T$ with $e$ epi such that $f \circ y = x \circ e$
\end{proposition}

Using this characterisation, the exactness of a sequence $U \xrightarrow{f} V \xrightarrow{g} W$ can be translated into the operator statement
\begin{align}\label{eq:exactness}
	g\cdot f \approx 0 \wedge \forall x \exists y, e \forall h :  g\cdot x \approx 0 \rightarrow \left( f\cdot y \approx x\cdot e \wedge \left(h\cdot e \approx 0 \rightarrow h \approx 0\right) \right),
\end{align}
where we have omitted the sorts of the variables for better readability.

Building upon exact sequences, also other common properties of morphisms in abelian categories can be encoded in our language, 
such as, being a mono- or epimorphism, or being a kernel or cokernel.
This allows to treat and efficiently prove many classical diagram lemmas with our framework, for example, the following 
\emph{Four lemma}.

\begin{lemma}[Four Lemma]
\label{lemma:four-lemma}
Let
\begin{center}
\begin{tikzpicture}[baseline=3]
    \matrix (m) [matrix of math nodes, row sep=9mm, column sep=8mm]
       {A & B & C & D\\
      A' & B' & C' & D' \\};
    \path[->] (m-1-1) edge node [auto] {$f$} (m-1-2);
    \path[->] (m-1-2) edge node [auto] {$g$} (m-1-3);
    \path[->] (m-1-3) edge node [auto] {$h$} (m-1-4);
    \path[->] (m-2-1) edge node [below] {$f'$} (m-2-2);
    \path[->] (m-2-2) edge node [below] {$g'$} (m-2-3);
    \path[->] (m-2-3) edge node [below] {$h'$} (m-2-4);
    \path[->] (m-1-1) edge node [auto] {$\alpha$} (m-2-1);
    \path[->] (m-1-2) edge node [auto] {$\beta$} (m-2-2);
    \path[->] (m-1-3) edge node [auto] {$\gamma$} (m-2-3);
    \path[->] (m-1-4) edge node [auto] {$\delta$} (m-2-4);
  \end{tikzpicture}
\end{center}
be a commutative diagram with exact rows.
If $\beta, \delta$ are monic and $\alpha$ is epi, then $\gamma$ is also monic.
\end{lemma}

Using an adaptation of Procedure~\ref{algo semantic consequence}, which we have implemented in a
\textsc{Mathematica} package \texttt{OperatorGB}\footnote{available at \url{https://github.com/ClemensHofstadler/OperatorGB}},
the Four lemma can be proven fully automatically in less than 20 seconds.
We refer to~\cite[Sec.~7.3]{Hof23} for further details on the computation together with instructions on how to use the software.

\section*{Acknowledgements}

We thank the anonymous reviewer of a first version of this paper for his comprehensive feedback, which greatly helped to improve the presentation of this work.
The first author was supported by the Austrian Science Fund (FWF) grant P32301.
The second author was supported by the Austrian Science Fund (FWF) grant P31952.

\bibliography{References} 
\bibliographystyle{plain}

\appendix
\section{Many-sorted logic}
\label{sec:many-sorted-logic}

Classical first-order logic is an established logic, satisfying many important theorems such as the Löwenheim-Skolem theorem~\cite[Thm.~VI.1.1]{EFT94} or the compactness theorem~\cite[Thm.~VI.2.1]{EFT94}.
It also possesses a sound and complete calculus, see, for example,~\cite[Ch.~IV]{EFT94}.
Classical first-order logic can be considered unsorted, as the structures used only consist of a single universe or domain of objects.
In mathematics or computer science, however, we often formalise statements involving different kinds of objects.
The goal of \emph{many-sorted first-order logic} is to provide a natural setting for modelling such situations.

In the following sections, we recall the syntax and semantics of many-sorted first-order logic, and we discuss a sound and complete calculus for it.
Our main reference for this part is~\cite{Man93}, but we also include notation from~\cite{EFT94}.

\subsection{Syntax}

The syntax of a logic describes the allowed symbols of the language and how they can be composed to construct (well-formed) terms and formulas.
In case of many-sorted first-order logic, we fix the following basic sets that form our alphabet.
We assume that all these sets are pairwise disjoint.

We fix a denumerable set $\Sort = \{s_1,s_2,\dots \}$ of \emph{sorts}.
Each sort can describe one type of object we want to model.
We also fix denumerable sets $\Con = \{c_1,c_2,\dots\}$ and $\Fun = \{f_1,f_2,\dots\}$ of 
\emph{constant symbols} and \emph{function symbols} respectively.
These symbols allow to model dedicated constants in structures as well as functions on them.
Note that we do not fix any predicate symbols as we will only work with one special predicate symbol $\approx$ for equality, which we will interpret as identity.
For a description of many-sorted first-order logic including arbitrary predicate symbols, we refer to~\cite{Man93}.
Finally, we fix a denumerable set $\Var = \{x_1,x_2,\dots\}$ of \emph{variables}. 

A \emph{signature} lists and describes the non-logical symbols of a logic that are relevant in a particular context.
In many-sorted logic, a signature is defined as follows.

\begin{definition}\label{def:signature-logic}
A \emph{signature} is a tuple $\Sigma = (\bm{S},\bm{C},\bm{F},\sigma)$ consisting of
\begin{enumerate}
	\item a set of sorts $\bm{S} \subseteq \Sort$;
	\item a set of constant symbols $\bm{C} \subseteq \Con$;
	\item a set of function symbols $\bm{F} \subseteq \Fun$;
	\item a \emph{sort function} $\sigma\colon \Var \cup \bm{C} \cup \bm{F} \to \bigcup_{n \geq 1} \bm{S}^n$ satisfying the following conditions:
	\begin{enumerate}
	\item $\sigma(x) \in \bm{S}$ for all $x \in \Var \cup \bm{C}$;\label{cond:sigma-1}
	\item $\sigma(f) \in \bigcup_{n > 1} \bm{S}^n$ for all $f \in \bm{F}$;\label{cond:sigma-2}
	\end{enumerate}
\end{enumerate}
\end{definition}

The set $\bm{S}$ defines the relevant sorts of the signature.
The sets $\bm{C}$ and $\bm{F}$ determine which constant and function symbols the signature contains.
The sort function $\sigma$ assigns to every variable and to all relevant symbols their \emph{sort}.
Conditions~\eqref{cond:sigma-1} and~\eqref{cond:sigma-2} ensure that variables and constant symbols get mapped to a single sort $s \in \bm{S}$,
while function symbols get mapped to a tuple of sorts $(s_1,\dots,s_n,s) \in \bm{S}^{n+1}$.

\begin{remark}\label{remark:function-symbols}
For a function symbol $f$, we write $\sigma(f) = s_1 \times \dots \times s_n \to s$ instead of $\sigma(f) = (s_1,\dots,s_n,s)$. 
The quantity $n$ is called the \emph{arity} of $f$.
\end{remark}

In the following, we fix a signature $\Sigma = (\bm{S},\bm{C},\bm{F},\sigma)$.
Within this signature, we can construct \emph{terms} and \emph{formulas}.

\begin{definition}
The set $\Term(\Sigma,s)$ of \emph{terms of sort $s \in \bm{S}$} is the smallest set satisfying the following conditions:
\begin{enumerate}
	\item $\{x \in \Var \cup \bm{C} \mid \sigma(x) = s\} \subseteq \Term(\Sigma,s)$;
	\item if $f \in \bm{F}$ with $\sigma(f) = s_1 \times \dots \times s_n \to s$ and $t_i \in \Term(\Sigma,s_{i})$ for $i = 1,\dots,n$,
	then $f(t_1,\dots,t_n) \in \Term(\Sigma,s)$;
\end{enumerate}

For a term $t \in \Term(\Sigma,s)$, its \emph{sort} is $\sigma(t) \coloneqq s$.
The set $\Term(\Sigma)$ of \emph{terms} is given by
\[
	\Term(\Sigma)\coloneqq \bigcup_{s \in \bm{S}} \Term(\Sigma,s).
\]
\end{definition}

A term is called \emph{ground} if it does not contain any variables.
The set of all ground terms is denoted by $\Ground(\Sigma)$.
Note that this set is countable.

\begin{example}[label=ex:terms]
Consider the signature $\Sigma = (\bm{S},\bm{C},\bm{F},\sigma)$ with $\bm{S} = \{s_{1},s_{2}\}$, $\bm{C} = \{c,d\}$, $\bm{F} = \{f,g\}$, and $\sigma$ such that
\[
	\sigma(c) = s_{1}, \qquad \sigma(d) = s_{2}, \qquad \sigma(f) = s_{1} \times s_{1} \to s_{2}, \qquad \sigma(g) = s_{2} \to s_{1}.
\]
Furthermore, let $x,y \in \Var$ with $\sigma(x) = s_{1}$ and $\sigma(y) = s_{2}$.

Within this signature, the expression $f(x,y)$ is \emph{not} a well-formed term because $y$ has the wrong sort.
Examples of well-formed terms are all the basic symbols $c,d,x,y$, but also $f(x,x)$ or $g(f(g(d), c))$.
Note that $c, d$ as well as $g(f(g(d), c))$ are ground terms, while $f(x,x)$ is not.
\end{example}

\begin{definition}\label{def:formulas}
The set $\Formulas(\Sigma)$ of \emph{formulas} is the smallest set satisfying the following conditions:
\begin{enumerate}
	\item if $t,t' \in \Term(\Sigma,s)$ for some $s \in \bm{S}$, then $t \approx t' \in \Formulas(\Sigma)$;\label{item:formulas-1}
	\item if $\varphi \in \Formulas(\Sigma)$, then $\lnot \varphi \in \Formulas(\Sigma)$;
	\item if $\varphi, \psi \in \Formulas(\Sigma)$, then $(\varphi \ast \psi) \in \Formulas(\Sigma)$ with $\ast \in \{\vee, \wedge, \rightarrow\}$;
	\item if $\varphi \in \Formulas(\Sigma)$ and $x \in \Var$, then $Q x : \varphi \in \Formulas(\Sigma)$ with $Q \in \{\exists, \forall\}$;
\end{enumerate}
\end{definition}

For convenience in notation, we follow the common convention to drop the outermost parentheses of formulas, and we assume the following precedence of the logical operators:
$\lnot, \vee, \wedge, \rightarrow, \exists, \forall$.
Furthermore, we abbreviate a block of consecutive equally quantified variables $Qx_1Qx_2\dots Qx_k$ with $Q \in \{\exists,\forall\}$ by $Qx_1,x_2,\dots, x_k$, or simply by $Q\bm{x}$.
Furthermore, to indicate the scope of a quantifier, we also write $Q \bm{x} : \varphi(\bm{x})$.

We recall some standard definitions from (many-sorted) first-order logic. 
Formulas constructed in part~\ref{item:formulas-1} of Definition~\ref{def:formulas} are called \emph{atomic formulas}.
A \emph{literal} is an atomic formula or its negation.
In the following, we write $t \not\approx t'$ for the literal $\lnot (t \approx t')$.
In the last part of Definition~\ref{def:formulas}, $Q$ is called the \emph{quantifier} of $x$ and $\varphi$ is the \emph{scope} of $Qx$.
Any occurrence of the variable $x$ in the scope $\varphi$ is called \emph{bound}.
All non-quantified occurrences of a variable, that is, all occurrences which are not in the scope of a quantifier, are called \emph{free} occurrences.
A variable $x$ is \emph{bound} (resp.~\emph{free}) in a formula $\varphi$ if there is a bound (resp.~free) occurrence of $x$ in $\varphi$.
We denote the set of all free variables in a formula $\varphi$ by $\Free(\varphi)$.
If $\Free(\varphi) = \emptyset$, then $\varphi$ is \emph{sentence}.
The set of all sentences is denoted by $\Sent(\Sigma)$.
Furthermore, a formula without any quantifiers is called \emph{quantifier-free}, and a quantifier-free sentence is called \emph{ground}.
Note that a ground sentence does not contain any variables.

\begin{example}[label=ex:formulas]
We reconsider the signature $\Sigma = (\bm{S},\bm{C},\bm{F},\sigma)$ from Example~\ref{ex:terms}, that is,
$\bm{S} = \{s_{1},s_{2}\}$, $\bm{C} = \{c,d\}$, $\bm{F} = \{f,g\}$, and $\sigma$ such that
\[
	\sigma(c) = s_{1}, \qquad \sigma(d) = s_{2}, \qquad \sigma(f) = s_{1} \times s_{1} \to s_{2}, \qquad \sigma(g) = s_{2} \to s_{1}.
\]
Furthermore, we let $x,y \in \Var$ with $\sigma(x) = s_{1}$ and $\sigma(y) = s_{2}$.

Within this signature, the expression $f(x,x) \approx c$ is \emph{not} a well-formed formula because $f(x,x)$ and $c$ have different sorts.
Examples of well-formed formulas are $f(x,x) \approx f(g(d), c)$ or $\forall x : \left( x \not\approx c \vee \exists y : g(y) \approx c \right)$.
The first formula is a quantifier-free atomic formula.
Note that it contains the free variable $x$, and is therefore not a sentence.
The second formula is a sentence because both variables that appear in it are bound, however it is not ground as it is not quantifier-free.
An example of a ground sentence is $c \approx g(d) \rightarrow f(c,g(d)) \approx f(g(d),c)$.
\end{example}

We denote by $\varphi[x \mapsto t]$ the substitution of the variable $x$ by the term $t$ in the formula~$\varphi$.
Such a substitution is only possible if $x$ and $t$ are of the same sort.
The result of this substitution is the formula obtained by replacing all free occurrences of $x$ by~$t$, with renaming, if necessary, those bound variables in $\varphi$
that coincide with a variable appearing in $t$. 
More precisely, we have the following formal definition of the substitution process, which we first specify for terms and then extend to formulas.

\begin{definition}
Let $x \in \Var$ and $t \in \Term(\Sigma)$ such that $\sigma(x) = \sigma(t)$.
The \emph{substitution} of $x$ by $t$ in a term is recursively defined as follows, where $y \in \Var$, $c \in \bm{C}$, and $f \in \bm{F}$:
\begin{align*}
	y[x \mapsto t] &\quad\coloneqq\quad \begin{cases} t &\text{if }x = y \\
						y	& \text{otherwise}
				\end{cases} \\
	c[x \mapsto t] &\quad\coloneqq\quad c \\
	f(t_1,\dots,t_n)[x \mapsto t] &\quad\coloneqq\quad f(t_1[x \mapsto t],\dots,t_n[x \mapsto t])
\end{align*} 
The \emph{substitution} of $x$ by $t$ in a formula is recursively defined as:
\begin{align*}
	(s \approx s')[x \mapsto t] &\quad\coloneqq\quad s[x \mapsto t] \approx s'[x \mapsto t] \\
	(\lnot \varphi)[x \mapsto t] &\quad\coloneqq\quad \lnot (\varphi[x \mapsto t]) \\
	(\varphi \ast \psi)[x \mapsto t] &\quad\coloneqq\quad \varphi[x \mapsto t] \ast \psi[x \mapsto t]\\
	(Qy : \varphi)[x \mapsto t] &\quad\coloneqq\quad \begin{cases} 
										Qy : \varphi &\text{if }x \notin \Free(Qy : \varphi) \\
										Qy : \varphi[x \mapsto t]  &\text{if }x \in \Free(Qy : \varphi), y \notin \Free(t)\\
										Qz : (\varphi[y \mapsto z])[x \mapsto t] &\text{if }x \in \Free(Qy : \varphi), y \in \Free(t),\\[-0.4em]
														&\text{and }z\text{ is a new variable with}\\[-0.4em]
														&\sigma(z) = \sigma(y)
									\end{cases}
\end{align*}
Here, $t_1,\dots,t_n,s,s' \in \Term(\Sigma)$, $\varphi, \psi \in \Formulas(\Sigma)$, $\ast \in \{\vee, \wedge, \rightarrow\}$, and $Q \in \{\exists,\forall\}$.
\end{definition}

\begin{example}
Consider a signature $\Sigma$ containing a function symbol $f \colon s_{1} \times s_{2} \to s_{1}$.
Furthermore, let $x,y,z \in \Var$ be such that $\sigma(x) = \sigma(y) = s_{1}$ and $\sigma(z) = s_{2}$.
We have, for example, 
\begin{enumerate}
	\item $(x \approx f(y,z))[x \mapsto f(x,z)] = f(x,z) \approx f(y,z)$;
	\item $\left(\exists x: x \approx f(y,z)\right)[x \mapsto f(x,z)] = \exists x: x \approx f(y,z)$;
	\item $\left(\exists y: x \approx f(y,z)\right)[x \mapsto f(x,z)] = \exists y: f(x,z) \approx f(y,z)$;
	\item $\left(\exists z: x \approx f(y,z)\right)[x \mapsto f(x,z)] = \exists z': f(x,z) \approx f(y,z')$, with $z'$ such that $\sigma(z') = \sigma(z) = s_{2}$;
\end{enumerate}
\end{example}

\subsection{Semantics}
\label{sec:preliminaries-semantics}

In ordinary first-order logic, formulas are assigned truth values by an \emph{interpretation} in a universe $A$.
In many-sorted logic, we no longer have a single universe but a dedicated domain $A_s$ for each sort $s \in \bm{S}$.
These domains, together with constants and functions that interpret the constant and function symbols, form a \emph{structure}.
Recall that we fixed a signature $\Sigma = (\bm{S},\bm{C},\bm{F},\sigma)$.

\begin{definition}
A \emph{structure} $A$ of signature $\Sigma$ is a tuple
\[
	A = \left((A_s)_{s \in \bm{S}}, (c^A)_{c \in \bm{C}}, (f^A)_{f \in \bm{F}}\right)
\]
satisfying the following conditions:
\begin{enumerate}
	\item $(A_s)_{s \in \bm{S}}$ is a family of nonempty sets;
	\item $ (c^A)_{c \in \bm{C}}$ is a family of elements such that $\sigma(c) = s$ implies $c^A \in A_s$;
	\item  $(f^A)_{f \in \bm{F}}$ is a family of functions such that $\sigma(f) = s_1 \times \dots \times s_n \to s$ implies
	\[
		f^A\colon A_{s_1} \times \dots \times A_{s_n} \to A_{s};
	\]
\end{enumerate}
\end{definition}

Given a structure $A$ of signature $\Sigma$, an \emph{assignment} on $A$ is a function 
\[
	\assign \colon \Var \to \bigcup_{s \in \bm{S}} A_{s}
\]
such that $\assign(x) \in A_{s}$ if and only if $\sigma(x) = s$ for all $x \in \Var$.

\begin{definition}
Let $A$ be a structure of signature $\Sigma$ and let $\assign$ be an assignment on $A$.
The pair $\mathcal{I} = (A,\assign)$ is called an \emph{interpretation}.
\end{definition}

An interpretation allows to map terms to elements in the different universes.
More precisely, the interpretation $\mathcal{I}(t)$ of a term $t \in \Term(\Sigma)$ is defined as
\[
	\mathcal{I}(t) \coloneqq
	\begin{cases} 
		\assign(x) & \text{ if }t = x \in \Var \\
		c^A & \text{ if }t = c \in \bm{C} \\
		f^A(\mathcal{I}(t_1),\dots,\mathcal{I}(t_n)) & \text{ if }t = f(t_1,\dots,t_n) \text{ with } f \in \bm{F}
	\end{cases}
\]
Note that all expressions above are well-defined.
In particular, if $t$ is of sort $s$, then $\mathcal{I}(t)$ is an element in $A_s$.

Using interpretations of terms, we can interpret formulas.
To do this, we introduce the following notation.
If $\assign$ is an assignment on a structure $A$, $x \in \Var$ is a variable of sort $s$ and $a \in A_{s}$, then
$\assign[x \mapsto a]$ denotes the function
\[
	\assign[x \mapsto a]\colon \Var \to \bigcup_{s \in \bm{S}} A_{s},
\]
that maps every $y \neq x$ to $\assign(y)$ and $x$ to $a$.
This function is clearly also an assignment on~$A$.
Furthermore, for $\mathcal{I} = (A,\assign)$, let $\mathcal{I}[x \mapsto a] = (A, \assign[x \mapsto a])$.

With this, the interpretation of a formula $\varphi \in \Formulas(\Sigma)$ is the truth value $\mathcal{I}(\varphi) \in \{\top,\bot\}$ defined as follows:
\begin{align*}
	 \mathcal{I}(t \approx t') &\coloneqq \top &&\text{iff} &&\mathcal{I}(t) = \mathcal{I}(t') \text{ as elements in } A_{\sigma(t)}\\
	 \mathcal{I}(\lnot \varphi) &\coloneqq \top &&\text{iff} &&\mathcal{I}(\varphi) = \bot\\
	 \mathcal{I}(\varphi \vee \psi) &\coloneqq \top &&\text{iff} &&\mathcal{I}(\varphi) = \top \text{ or } \mathcal{I}(\psi) = \top\\
	 \mathcal{I}(\varphi \wedge \psi) &\coloneqq \top &&\text{iff} &&\mathcal{I}(\varphi) = \top \text{ and } \mathcal{I}(\psi) = \top\\
	 \mathcal{I}(\varphi \rightarrow \psi) &\coloneqq \top &&\text{iff} &&\mathcal{I}(\varphi) = \bot \text{ or } \mathcal{I}(\psi) = \top\\
	 \mathcal{I}(\forall x : \varphi) &\coloneqq \top &&\text{iff} &&\text{for all } a \in A_{\sigma(x)} \text{ it holds that } \mathcal{I}[x \mapsto a](\varphi) = \top\\
	 \mathcal{I}(\exists x : \varphi) &\coloneqq \top &&\text{iff} &&\text{there exists some } a \in A_{\sigma(x)} \text{ such that } \mathcal{I}[x \mapsto a](\varphi) = \top
\end{align*}

\begin{remark}\label{remark:equality}
The equality symbol $\approx$ is always interpreted as the identity in the structure $A$.
Such structures that interpret equality as identity are called \emph{normal}.
We only consider normal structures in this work.
\end{remark}

An interpretation $\mathcal{I}$ is a \emph{model} of a formula $\varphi$ if $\mathcal{I}(\varphi) = \top$.
Furthermore, $\varphi$ is \emph{valid} if every interpretation is a model of $\varphi$.
Two formulas $\varphi$ and~$\psi$ are \emph{logically equivalent} if they have the same models.

Similarly, an interpretation is a model of a set of formulas $\Phi \subseteq \Formulas(\Sigma)$ if it is a model of every formula in $\Phi$,
and $\Phi$ is \emph{valid} if every formula in $\Phi$ is valid.

A formula $\varphi$ is a \emph{semantic consequence} of a set of formulas $\Phi$, if each model of $\Phi$ is also a model of $\varphi$.
In this case, we write $\Phi \models \varphi$.
If $\varphi$ is not a semantic consequence of $\Phi$, we write $\Phi \not\models \varphi$. 
Two formulas $\varphi$ and $\psi$ are \emph{$\Phi$-equivalid} if $\Phi \models \varphi$ if and only if $\Phi \models \psi$.

%

\subsection{Formal computations}
\label{sec:formal-comp}

In this section, we recall how to validate that a formula $\varphi$ is a semantic consequence of a set of formulas $\Phi$ by performing syntactic operations.
Such a syntactic computation is called a \emph{formal computation} and one way to perform it is by applying certain syntactic rules, so-called \emph{sequent rules}, to the formulas at hand.
Several sequent rules form a \emph{sequent calculus}.

We recall the sequent calculus LK$^=$ for many-sorted first-order logic with equality.
This calculus -- for ordinary first-order logic without equality -- was first introduced by Gentzen~\cite{Gen35} and has subsequently been extended to include equality,
see, e.g.,~\cite{DV98} for further information in the unsorted case.
As illustrated in~\cite[Sec.~2]{Man93}, the adaptation to the many-sorted case is straightforward.

We note that there also exist other inference systems for (many-sorted) first-order logic.
We focus on the sequent calculus LK$^=$ in this work as this inference system allows a natural translation of formal computations with formulas into algebraic computations with noncommutative polynomials (see Section~\ref{sec:arithmetic-statements}).
In particular, the specific structure of the sequent calculus, providing a separate treatment for each logical connective, allows us to prove Lemma~\ref{lemma preserve contradiction}.

A \emph{sequent} is an expression of the form
$\Gamma \vdash \Delta$
where $\Gamma$ and $\Delta$ are finite (possibly empty) multisets of formulas.
A sequent can be considered as an assertion of the form ``whenever \emph{all} $\varphi \in \Gamma$ are true, then \emph{at least one} $\psi \in \Delta$ is also true''.
An expression of the form $\Gamma, \varphi$ in a sequent represents the multiset $\Gamma \cup \{\varphi\}$.
In particular, $\varphi_1,\dots,\varphi_m \vdash \psi_1,\dots,\psi_n$ means $\{\varphi_1,\dots,\varphi_m\} \vdash \{\psi_1,\dots,\psi_n\}$.

\emph{Sequent rules} allow us to pass from one sequent to another.
The sequent calculus LK$^=$ comprises the following sequent rules.
In the following, $\alpha$ denotes an atomic formula, $\varphi, \psi$ are arbitrary formulas, and $t,t'$ are terms.
\begin{alignat*}{4}
\intertext{\textbf{Axioms}}
&\text{(Ax) }
\centerAlignProof
&&\AxiomC{}
\def\fCenter{\ \vdash\ }
\UnaryInfC{$\Gamma, \alpha \fCenter \Delta, \alpha$}
\DisplayProof
\hskip 4em
&&\text{(Ref) }
\centerAlignProof
&&\AxiomC{}
\def\fCenter{\ \vdash\ }
\UnaryInf$\Gamma \fCenter \Delta, t \approx t$
\DisplayProof
\intertext{\textbf{Structural rules}}
&
\text{($\text{W}\ \vdash)$\;\; }
&&
\def\fCenter{\ \vdash\ }
\Axiom$\Gamma \fCenter \Delta$
\UnaryInf$\Gamma, \varphi \fCenter \Delta$
\DisplayProof
\hskip 4em
&&
\text{($\vdash\ \text{W})$\;\; }
&&
\def\fCenter{\ \vdash\ }
\Axiom$\Gamma \fCenter \Delta$
\UnaryInf$\Gamma \fCenter \Delta, \varphi$
\DisplayProof
\\[2em]
&
\text{($\text{C}\ \vdash)$ }
&&\def\fCenter{\ \vdash\ }
\Axiom$\Gamma, \varphi, \varphi \fCenter \Delta$
\UnaryInf$\Gamma, \varphi \fCenter \Delta$
\DisplayProof
\hskip 4em
&&
\text{($\vdash\ \text{C})$ }
&&\def\fCenter{\ \vdash\ }
\Axiom$\Gamma \fCenter \Delta, \varphi, \varphi$
\UnaryInf$\Gamma \fCenter \Delta, \varphi$
\DisplayProof
\intertext{\textbf{Propositional rules}}
&
\text{($\lnot\ \vdash)$ }
&&\def\fCenter{\ \vdash\ }
\Axiom$\Gamma \fCenter \Delta, \varphi$
\UnaryInf$\Gamma, \lnot \varphi \fCenter \Delta$
\DisplayProof
\hskip 4em
&&
\text{($\vdash\ \lnot)$ }
&&\def\fCenter{\ \vdash\ }
\Axiom$\Gamma, \varphi \fCenter \Delta$
\UnaryInf$\Gamma \fCenter \Delta, \lnot \varphi$
\DisplayProof
\\[2em]
&
\text{($\vee\ \vdash)$ }
&&\def\fCenter{\ \vdash\ }
\Axiom$\Gamma, \varphi \fCenter \Delta$
\Axiom$\Gamma, \psi \fCenter \Delta$
\BinaryInf$\Gamma, \varphi \vee \psi \fCenter \Delta$
\DisplayProof
\hskip 4em
&&
\text{($\vdash\ \vee)$ }
&&\def\fCenter{\ \vdash\ }
\Axiom$\Gamma \fCenter \Delta, \varphi, \psi$
\UnaryInf$\Gamma \fCenter \Delta, \varphi \vee \psi$
\DisplayProof
\\[2em]
&
\text{($\wedge\ \vdash)$ }
&&\def\fCenter{\ \vdash\ }
\Axiom$\Gamma, \varphi, \psi \fCenter \Delta$
\UnaryInf$\Gamma, \varphi \wedge \psi \fCenter \Delta$
\DisplayProof
\hskip 4em
&&
\text{($\vdash\ \wedge)$ }
&&\def\fCenter{\ \vdash\ }
\Axiom$\Gamma \fCenter \Delta, \varphi$
\Axiom$\Gamma \fCenter \Delta, \psi$
\BinaryInf$\Gamma \fCenter \Delta, \varphi \wedge \psi$
\DisplayProof
\\[2em]
&
\text{($\rightarrow\ \vdash)$ }
&&\def\fCenter{\ \vdash\ }
\Axiom$\Gamma \fCenter \Delta, \varphi$
\Axiom$\Gamma, \psi \fCenter \Delta$
\BinaryInf$\Gamma, \varphi \rightarrow \psi \fCenter \Delta$
\DisplayProof
\hskip 4em
&&
\text{($\vdash\ \rightarrow)$ }
&&\def\fCenter{\ \vdash\ }
\Axiom$\Gamma, \varphi \fCenter \Delta, \psi$
\UnaryInf$\Gamma \fCenter \Delta, \varphi \rightarrow \psi$
\DisplayProof
\end{alignat*}

\begin{alignat*}{2}
\intertext{\textbf{Quantification rules}}
&
\text{($\exists\ \vdash)$\;\; }
\RightLabel{if $y$ is a new variable with $\sigma(y) = \sigma(x)$}
&&\def\fCenter{\ \vdash\ }
\Axiom$\Gamma, \varphi[x \mapsto y] \fCenter \Delta$
\UnaryInf$\Gamma, \exists x : \varphi \fCenter \Delta$
\DisplayProof
\\[2em]
&
\text{($\vdash\ \exists)$ }
\RightLabel{if $t$ is a term with $\sigma(t) = \sigma(x)$}
&&\def\fCenter{\ \vdash\ }
\Axiom$\Gamma \fCenter \Delta, \varphi[x \mapsto t]$
\UnaryInf$\Gamma \fCenter \Delta, \exists x : \varphi$
\DisplayProof
\\[2em]
&
\text{($\forall\ \vdash)$ }
\RightLabel{if $t$ is a term with $\sigma(t) = \sigma(x)$}
&&\def\fCenter{\ \vdash\ }
\Axiom$\Gamma, \varphi[x \mapsto t] \fCenter \Delta$
\UnaryInf$\Gamma, \forall x : \varphi \fCenter \Delta$
\DisplayProof
\\[2em]
&
\text{($\vdash\ \forall)$ }
\RightLabel{if $y$ is a new variable with $\sigma(y) = \sigma(x)$}
&&\def\fCenter{\ \vdash\ }
\Axiom$\Gamma \fCenter \Delta, \varphi[x \mapsto y]$
\UnaryInf$\Gamma \fCenter \Delta, \forall x  : \varphi$
\DisplayProof
\intertext{\textbf{Equational rule}}
&
\text{(Sub) }
&&\def\fCenter{\ \vdash\ }
\RightLabel{if $\sigma(x) = \sigma(t)$}
\Axiom$\Gamma, \varphi[x \mapsto t'], t \approx t' \fCenter \Delta$
\UnaryInf$\Gamma, \varphi[x \mapsto t], t \approx t' \fCenter \Delta$
\DisplayProof
\end{alignat*}

A rule without a sequent on top is called an \emph{axiom}.
Any statement that can be derived from the axioms using the rules stated above, is \emph{provable by a formal computation} (cf.~\cite[Def.~IV.1.1]{EFT94}).

\begin{definition}
A formula $\varphi$ is \emph{provable by a formal computation} from a set $\Phi$ of formulas
if there exists a finite subset $\{\varphi_1,\dots,\varphi_n\} \subseteq \Phi$ such that the sequent $\varphi_1,\dots,\varphi_n \vdash \varphi$ is derivable by the rules of the sequent calculus \textup{LK$^{=}$} stated above starting from axioms.
In this case, we write $\Phi \vdash \varphi$.
\end{definition}

The following theorem recalls that LK$^=$ is sound and complete.
Soundness refers to the fact that every formal computation yields a semantic consequence and completeness states that every semantic consequence is provable by a formal computation.
While the former follows easily from the definition of the sequent rules, the latter is essentially the result of \emph{G\"odel's completeness theorem}~\cite{godel1930vollstandigkeit}, see also~\cite[Sec.~2]{Man93} for the many-sorted case.

\begin{theorem}\label{thm gödel}
Let $\Phi \subseteq \Formulas(\Sigma)$ and $\varphi \in \Formulas(\Sigma)$. 
Then $\Phi \models \varphi$ if and only if $\Phi \vdash \varphi$.
\end{theorem}


\end{document}